\documentclass[a4paper,10pt]{amsart}
\usepackage{amsmath,amsthm,amssymb,tikz,calc,graphicx}
\usepackage[latin1]{inputenc}
\usepackage[T1]{fontenc}
\usepackage[all]{xy}
\usepackage{mathrsfs,pifont}
\usepackage{booktabs}
\usepackage{array}
\usepackage{longtable}
\usepackage{placeins}

\newtheorem{thm}[equation]{Theorem}

\newtheorem*{thm*}{Theorem}

\newtheorem{cor}[equation]{Corollary}
\newtheorem{lem}[equation]{Lemma}
\newtheorem{prop}[equation]{Proposition}

\theoremstyle{definition}

\newtheorem{rem}[equation]{Remark}

\newtheorem{defn-lem}[equation]{Definition-Lemma}
\newtheorem{rmk}[equation]{Remark}
\newtheorem*{rem*}{Remark}

\numberwithin{equation}{section}

\usepackage{bm}
\usepackage{upgreek}
\makeatletter
\newcommand{\bfgreek}[1]{\bm{\@nameuse{up#1}}}

\newcommand{\triv}{{\textbf{1}}}
\def\R{\mathbb R}

\def\Z{\mathbb Z}
\def\A{\mathbb A}
\def\Q{\mathbb Q}
\def\C{\mathbb C}

\def\EE{\mathcal E}
\def\cF{\mathcal F}
\def\Res{{\rm Res}}

\def\hra{\hookrightarrow}
\def\ra{\rightarrow}

\def\g{\mathfrak g}

\def\n{\mathfrak n}
\def\a{\mathfrak a}

\def\q{\mathfrak q}

\def\h{\mathfrak h}

\def\m{\mathfrak m}
\def\k{\mathfrak k}
\def\l{\mathfrak l}

\def\u{\mathfrak u}

\def\ww{\mathsf w}

\def\<{\langle}
\def\>{\rangle}

\def\J{\mathcal J}

\def\Hom{{\rm Hom}}

\def\GL{{\rm GL}}

\def\vv{{\sf v}} 

\makeatletter
\DeclareRobustCommand\smallop[2][1]{%
	\mathop{\vphantom{\oplus}\mathpalette\smallop@{{#1}{#2}}}\slimits@
}
\newcommand{\smallop@}[2]{\smallop@@#1#2}
\newcommand{\smallop@@}[3]{%
	\vcenter{%
		\sbox\z@{$#1\oplus$}%
		\hbox{\resizebox{\ifx#1\displaystyle#2\fi\dimexpr\ht\z@+\dp\z@}{!}{$\m@th#3$}}%
	}%
}
\makeatother

\makeatletter
\DeclareRobustCommand\bigop[2][1]{%
	\mathop{\vphantom{\bigoplus}\mathpalette\bigop@{{#1}{#2}}}\slimits@
}
\newcommand{\bigop@}[2]{\bigop@@#1#2}
\newcommand{\bigop@@}[3]{%
	\vcenter{%
		\sbox\z@{$#1\bigoplus$}%
		\hbox{\resizebox{\ifx#1\displaystyle#2\fi\dimexpr\ht\z@+\dp\z@}{!}{$\m@th#3$}}%
	}%
}
\makeatother

\newcommand{\bigtopdirsum}{%
	\tikz[baseline]{\draw[thick] 
		(-.025,-1ex) -- ++(0,3.2ex) 
		(.025,-1ex) -- ++(0,3.2ex) 
		(0,-.025)++(-1.6ex,.6ex) -- ++(3.2ex,0) 
		(0,.025)++(-1.6ex,.6ex) -- ++(3.2ex,0) 
		(0,.6ex) circle (1.6ex) }%
}

\newcommand{\bigtoplus}{\DOTSB\bigop[.86]{\bigtopdirsum}}

\begin{document}

\title{On the residual Eisenstein cohomology of unitary groups}
\author{Harald Grobner}
\thanks{The author has been supported by the research projects PAT-4628923 and PAT-2584625 of the Austrian Science Fund (FWF), the EU NextGeneration project (IIP\_UNIPU\_010159) and the Croatian Science Foundation project (HRZZ-IP-2022-10-4615).}
\address{Harald Grobner, Faculty of Mathematics, University of Vienna, Oskar-Morgenstern-Platz 1, A-1090 Vienna, Austria}
\email{harald.grobner@univie.ac.at}

\date{\today}

\maketitle

 
\begin{abstract}
We investigate the residual Eisenstein cohomology of an arbitrary unitary group $U(V)$ attached to an arbitrary quadratic extension of number fields $E/F$. Our focus lies on the contribution of the maximal parabolic $F$-subgroups of $U(V)$, for which we identify the cohomologically relevant poles of Eisenstein series and prove that the resulting residues all survive as non-trivial classes in automorphic cohomology in an explicit degree. To illustrate the range of phenomena involved, we study in detail the case of a unitary group over $F=\Q(\sqrt[3]{2})$ of $F$-rank $3$, for which we explicitly construct cuspidal automorphic representations, which satisfy all the assumptions of our main theorem and hence explicitly construct non-zero residual Eisenstein cohomology classes for this unitary group. The methods used in this construction are paradigmatic however, i.e., generalize to other unitary groups over other ground fields $F$ by the use of base change.
\end{abstract}

\vskip 10pt

\setcounter{tocdepth}{2}
\tableofcontents

\section*{Introduction}

\noindent The cohomology of arithmetic groups is a focal point of the global Langlands program. Here, one not only sees the program's key assertion -- namely, that Galois theory of number fields $F$, the geometry of motives, and the representation theory of automorphic forms are just three sides of the same (miraculous) coin -- demonstrated with remarkable clarity and (in the author's opinion) even simplicity, but even gains insight into {\it why} these symmetries of the theory should and must exist.\\\\
If the considered reductive group $G/F$ gives rise to non-compact such arithmetic quotients, i.e., if its $F$-rank is positive, then the existence of non-compactly supported cohomology classes reflects a deep principle: Geometrically speaking, it says that there must be algebraic differential forms, which restrict non-trivially to the boundary of the Borel-Serre compactification; automorphically, it says that there must be non-cuspidal cohomology classes, constructed via Eisenstein series. Again, both perspectives are intimately linked: Performing the summation process when forming an Eisenstein series just corresponds geometrically to lifting a class on the boundary to the whole compactified arithmetic quotient by performing the obvious averaging.\\\\ In some sense closest to cuspidal classes are those Eisenstein cohomology classes, which are still representable by square-integrable automorphic forms (just as the cuspidal forms are), i.e., those generated by residues of Eisenstein series, so called {\it residual Eisenstein cohomology classes}. By results of Borel, \cite{bor1}, the constant functions serve as a trivial example of such square-integrable, yet, residual forms. But while the contribution of constant forms to automorphic cohomology is relatively (!) well understood, cf.\ \cite{franke2}, the contribution of non-constant residual automorphic representations is much more delicate and still in large parts mysterious (nevertheless, see \cite{grob22}, \cite{rosp}, \cite{grobner-EisRes} and the recent \cite{mundy} for general results): Determining whether or not such automorphic residues actually give non-zero cohomology classes in automorphic cohomology requires a {\it simultaneous} control of the poles of Eisenstein series and of the restriction of the resulting classes to the Borel--Serre boundary, i.e., knowledge of whether or not their relevant constant terms still produce non-zero cohomology classes. \\\\
This paper carries out this analysis for unitary groups in a form which is
uniform in the ground field. Let \(E/F\) be an arbitrary quadratic
extension of number fields and let
$$
        G_n=U(V_n)
$$
be the unitary group attached to some non-degenerate $n$-dimensional  hermitian \(E\)-space
\(V_n\). We assume that \(G_n\) has positive \(F\)-rank $rk_F(G_n)\geq 1$, as otherwise there cannot be any non-cuspidal automorphic forms, but we impose no further
restriction on $V_n$ (in particular, we impose no restrictions on $V_n$ at the archimedean places) and no restriction
on \(F\). For a standard maximal \(F\)-parabolic subgroup \(P_k\) with
Levi factor
$$
        L_k\cong \operatorname{Res}_{E/F}\GL_k\times G_{n-2k},
$$
we consider cuspidal automorphic representations (more precisely, their associate classes $\varphi_{P_k}$)
$$
        \pi=\tau\widehat\otimes\sigma
$$
on \(L_k(\mathbb A_F)\), where \(\tau\) is a unitary cuspidal automorphic
representation of \(\GL_k(\mathbb A_E)\), and \(\sigma\) is a unitary
cohomological cuspidal automorphic representation of
\(G_{n-2k}(\mathbb A_F)\) with generic base change $BC(\sigma)$, together with a point $\lambda_0$ in the closed positive Weyl chamber of the split component of $L_k$. 
\\\\
Our first result, see Thm.\ \ref{thm:Poles}, is a pole criterion in the cohomological range. Writing $\lambda_0$ in the basis given by the global absolute value of the determinant on the $\GL_k$-factor, the only residual values which occur in the range relevant for cohomology are at the coordinates
$$
        s=\tfrac12
        \qquad\text{and}\qquad
        s=1 .
$$
At \(s=\frac12\), the simple pole is governed by $\tau$ being conjugate self-dual and \(\epsilon\)-distinguished, and the non-vanishing of the central
Rankin--Selberg value
$$
        L\!\left(\tfrac12,\tau\times BC(\sigma)^\vee\right).
$$
At \(s=1\), the simple pole occurs when \(\tau\) is one of the
cuspidal isobaric summands of the base change \(BC(\sigma)\). In both cases the above conditions -- to be explained with care and all details left out here in Thm.\ \ref{thm:Poles} -- are necessary and sufficient, hence a characterization, for the existence of the pole at $\lambda_0$ of the Eisenstein series in question.\\\\
Our second result, which is our main theorem, shows that these poles are genuinely cohomological, i.e., the attached residues all represent non-zero cohomology classes in the space of automorphic cohomology with respect to a previously fixed, but arbitrary finite-dimensional coefficient module $\EE_\mu$ of $G_\infty$. More precisely, for the associate class $\varphi_{P_k}$, represented by $\pi$ and $\lambda_0$, we consider the degree
$$
        q(\pi_{\lambda_0}):
        =
        q_{min}(\pi_\infty)
        +[F:\mathbb Q]\,k(2n-3k)-\ell(w),
$$
where $\ell(w)$ is the length of the unique Kostant representative \(w\in W^{P_k}\) determined by $\pi_\infty$ and $\lambda_0$ and $q_{min}(\pi_\infty)$ denotes the minimal degree of non-zero cohomology of $\pi_\infty$. (The latter is made completely explicit in Prop.\ \ref{prop:qminLk}, to which we refer the reader, who is interested in a detailled but lengthy formula.) Our main theorem now says that whenever the pole criterion above is satisfied at \(\lambda_0\), the resulting residues of the Eisenstein series all give rise to {\it non-zero} residual Eisenstein cohomology classes in degree \(q(\pi_{\lambda_0})\), i.e., yield non-zero classes in the relative Lie algebra cohomology of the space of $\EE_\mu$-valued automorphic forms $\mathcal A(G)\otimes\EE_\mu$: More precisely, it shows that the natural map 
$$H^{q(\pi_{\lambda_0})}(\g,K_G,\mathcal A_{res,\varphi_{P_k}}(G)\otimes\EE_\mu)\longrightarrow H^{q(\pi_{\lambda_0})}(\g,K_G,\mathcal A(G)\otimes\EE_\mu)$$
from the cohomology of the direct sum $\mathcal A_{res,\varphi_{P_k}}(G)$ of {\it all} the residual automorphic representations supported in the associate class $\varphi_{P_k}$ of $\pi$ and $\lambda_0$ to the space of automorphic cohomology is {\it injective}. We refer to Thm.\ \ref{thm:main} for all details, omitted here for sake of brevity, and an entirely precise statement.\\\\
Our third result produces actual examples of cohomological cuspidal automorphic representations $\pi=\tau\widehat\otimes\sigma$, which satisfy all criteria and assumptions to be made, in order make our above main theorem applicable. Constructing cuspidal automorphic representation with prescribed local and global behaviour (in particular, if the local components in question are not in the discrete series) in an unconditional way is a delicate task, which we carry out explicitly in our role model case of a quasisplit unitary group $U(V_6)$ over the field $F=\Q(\sqrt[3]{2})$. We refer to Thm.\ \ref{prop:nonempty-P2-example} for this result, which provides and unconditionally constructs non-zero residual Eisenstein cohomology classes in degree $q(\pi_{\lambda_0})=14$ with respect to trivial coefficients $\EE_\mu=\C$. The example is not merely illustrative, but paradigmatic: It shows that the hypotheses of the general theorem are simultaneously realizable even in a genuinely mixed archimedean situation (as in this example $U(V_6)_\infty\cong U(3,3)\times \GL_6(\C)$). Moreover, the methods used in our construction, which go back to fundamental ideas of Clozel \cite{clozelSLN}, certainly generalize to other unitary groups over other algebraic ground fields $F$, using base change \cite{mok, KMSW, zou}.\\\\
Our paper should also be seen in the context of the existing literature on Eisenstein cohomology for unitary groups, which contains several important results, notably for low-rank families over \(\mathbb Q\): See \cite{harderGU} for Harder's foundational work on $GU(2,1)/\Q$, \cite{grabschwun} for \(U(n,1)/\mathbb Q\), \cite{hayataschwermer} for \(SU(2,2)/\mathbb Q\), and more recently \cite{Ber-vdG22} for work on Eisenstein cohomology of $GU(2,1)/\Q(\sqrt{-3})$,  \cite{nairrai} for work related to \(SU(n,1)/\mathbb Q\) and \cite{bajcav}, again for the case of $SU(2,1)/\Q$. In contrast, the present paper isolates the residual mechanism for maximal parabolic subgroups in a setting which allows an arbitrary quadratic extension \(E/F\) of number fields, arbitrary \(F\)-rank and unitary groups attached to arbitrary non-degenerate hermitian forms (including, of course, arbitrary behaviour at the archimedean places). In particular, our setup uses no {\it Shimura-type} or signature hypothesis, but properly subsumes all of the latter.\\\\
On a final note, one reason for keeping our construction as explicit as possible is the connection of residual Eisenstein cohomology classes with Harder-type Eisenstein congruences. After the rational strucutures on $L^2$-automorphic cohomology, introduced in \cite{nairrai}, have been refined to integral structures, such classes are natural candidates for congruences between residual and cuspidal cohomology classes. In this sense, our paper provides the complex input for a study of Eisenstein congruences for unitary groups in the spirit of Harder's approach. We hope to return to these arithmetic aspects in future work.\\\\
\small
\noindent{\it Acknowledgments.}\enlargethispage{1cm}
This paper was completed after the passing of G\"unter Harder in 2025. Since Eisenstein cohomology is inseparable from Harder's work on the cohomology of arithmetic groups, it seems appropriate to record our gratitude for the perspective opened by his ideas, which G\"unter continued to share with characteristic generosity.
\normalsize

\section{Preliminaries on unitary groups and notation}\label{sect:preliminaries}

\subsection{Number fields}\label{sect:nf}
Throughout this paper $F$ denotes an algebraic number field, i.e., an arbitrary finite extension of $\Q$, and $E/F$ is an arbitrary quadratic extension with non-trivial Galois automorphism denoted by $c$. We write $S=S(F)$ for the set of places $\vv$ of $F$ and $S_\infty=S_\infty(F)$ (resp.\ $S_f = S_f(F)$) for its subset of archimedean (resp.\ non-archimdean) places. It disintegrates into $S_\infty=S_\R\cup S_\C$, where $S_\R$ denotes the subset of real places $\vv$, corresponding to the real field embeddings $\iota_\vv: F\hra\R$ and $S_\C$ denotes the subset of complex places $\vv$, corresponding to the pairs $\{\iota_\vv,\overline{\iota}_\vv\}$ of complex-conjugate non-real field embeddings $\iota_\vv: F\hra\C$. Write $S^{\sf sp}_\R:=\{\vv\in S_\R \ | \ E\otimes_{F}F_\vv\cong \R\times\R\}$ resp.\ $S^{\sf in}_\R:=\{\vv\in S_\R \ | \ E\otimes_{F}F_\vv\cong \C\}$ for the subsets of real places, which split in $E$, resp.\ stay inert in $E$. We will write $\Sigma$ (resp.\ $\Sigma_\infty$, resp.\ $\Sigma_f$) for the set of (archimedean, resp.\ non-archimedean) places $\ww$ of $E$. Then $S^{\sf in}_\R$ naturally identifies with a subset of $\Sigma_\infty$. If $\vv$ (resp.\ $\ww$) is a place of $F$ (resp.\ $E$), then we write $|\cdot|_\vv$ (resp.\ $|\cdot|_\ww$) for the respective normalized absolute value. \\\\
Let $\A_F$ (resp. $\A_E$) be the locally compact topological ring of ad\'eles of $F$ (resp.\ $E$). The corresponding absolute value of $\A_F$ (resp. $\A_E$) is denoted $\|\cdot\|_F:=\prod_{\vv\in S}|\cdot|_\vv$ (resp.\ $\|\cdot\|_E:=\prod_{\ww\in\Sigma}|\cdot|_\ww$). We let $\varepsilon_{E/F}: F^*\backslash\A^*_F\ra\C^*$ be the quadratic Hecke character associated with $E/F$ by class field theory and, given some $n,k\geq 0$, we define a quadratic character of $\GL_k(\A_F)$ by 
$$\epsilon:=\left\{\begin{array}{ll}
 \varepsilon_{E/F}\circ\textrm{det}_{\GL_k} & \textrm{if $n$ is even} \\
 \triv_{\GL_k(\A_F)} & \textrm{if $n$ is odd,}
\end{array}
\right.$$
suppressing its dependence on $n,k$ in the notation. Here, $\triv_H$ stands for the trivial character of a given subgroup $H$ of $\GL_k(\A_E)$. Also, we may and will choose a unitary conjugate self-dual Hecke character \(\eta_{E/F} :E^*\backslash \A^*_E\ra\C^*\), which extends $\varepsilon_{E/F}$, i.e., such that $\left.\eta_{E/F}\right|_{\mathbb A_F^\times}=\varepsilon_{E/F}$. One may arrange that at all complex places $\ww\in\Sigma_\infty$ it is of the form $z\mapsto z^{u}\bar{z}^{-u}$ with $u\in\tfrac12+\Z$. This follows analogously to \cite[\S 6.9.2]{bel-chen}.

\subsection{Algebraic groups}\label{sect:paras} 
Let $(V_n,\<\cdot,\cdot\>)$ be a non-degenerate $c$-hermitian form on an $E$-vector space $V_n$ of dimension $n\geq 1$ and denote by
$$
G:=G_n:=U(V_n),
$$
the attached unitary group. It is a connected reductive linear algebraic group over $F$ whose $R$-points for any commutative $F$-algebra $R$ are given by
$$G(R)=\{g\in {\rm Aut}_{E\otimes_F R}(V\otimes_F R) \ | \ \langle gv, gw\rangle=\langle v,w\rangle \}\subset \GL_n(R).$$
We abbreviate $r:=rk_F(G)$ for its $F$-rank, which we assume to be non-zero, i.e., $V$ that is isotropic, in hindsight to the existence of non-cuspidal (e.g., residual) automorphic representations of $G(\A_F)$.\\\\ 
We fix once and for all a minimal parabolic $F$-subgroup $P_0$ of $G$ with with Levi decomposition $P_0=L_0N_0$ and let $A_0$ be the maximal $F$-split torus in the center $Z_{L_0}$ of $L_0$. This choice defines the standard parabolic $F$-subgroups $P$ with Levi decomposition $P=L_PN_P$, where $L_P\supseteq L_0$ and $N_P\subseteq N_0$. We let $A_P$ be the maximal $F$-split torus in the center $Z_{L_P}$ of $L_P$, satisfying $A_P\subseteq A_0$. Every such $F$-parabolic $P$ is the stabilizer of a flag of totally isotropic $E$-subspaces $\{0\}\subseteq X_1\subset... \subset X_{t_P}\subseteq V$ and choosing a compatible Witt decomposition $V=X\oplus V_{\ell_P}\oplus X^\vee$ over $E$, one hence gets 
\begin{equation}\label{eq:LP}
L_P\cong \prod_{i=1}^{t_P}\Res_{E/F}\GL_{k_i}\times U(V_{\ell_P})
\end{equation}
where $k_i=\dim_E X_i/X_{i-1}$ and $\ell_P=\dim_EV_{\ell_P}$, whence $n=\ell_P +2\sum_{i=1}^{t_P} k_i$. The number $t_P$ of factors of general linear groups here equals the parabolic rank of $P$, i.e., the dimension of $\check\a_P:=X^*(A_P)\otimes_\Z\R$ or, likewise, $\a_P:=X_*(A_P)\otimes_\Z\R$ as an $\R$-vector space, where $X^*$ (resp. $X_*$) denotes the group of $F$-rational characters (resp. co-characters). This entails that $A_G$ is trivial, which can also be seen be noting that the center of $G$ is $Z_G\cong U(V_1)=\ker (N_{E/F}: \Res_{E/F}(\mathbb G)\ra\mathbb G)$. Our assumption that $r=rk_F(G)\geq 1$ is hence equivalent to assuming that all arithmetic subgroups of $G(F)$ are non-cocompact, though of finite covolume in $\Res_{F/\Q}(G)(\R)$, cf.\ \cite{serrearith,
 borelhch, borelbook}.\\\\ 
 We fix once and for all a choice of a maximal compact subgroup $K_{\A_F}=\prod_{\vv\in S_F}K_\vv$ of $G(\A_F)$, which is in good position with respect to our fixed choice of standard parabolic $F$-subgroups, cf. \cite{moewal} I.1.4 for the standard (resp.\ \cite{grob_book}, \S 9.2 for a more detailed) reference. With respect to $K_{\A_F}$ we obtain for each parabolic $F$-subgroup $P$ the {\it Harish-Chandra height function} $H_P: G(\A_F)\ra\a_P$, cf.\ \cite{grob_book}, p.\ 120. The kernel $L(\A_F)^{(1)}$ of $H_P|_{L(\A)}$ is naturally complemented within $L(\A)$ by a connected, central Lie group, $A^\R_P\cong \R_{>0}^{t_P}$ with Lie algebra isomorphic to $\a_P$. The resulting quotients $A^\R_P L(F)\backslash L(\A)$ have finite invariant volume. See \cite{grob_book}, Prop.\ 9.11 for this claim and further, more original references.\\\\ 
We write $\{P\}$ for the {\it associate class of $P$}, i.e., the set of all standard parabolic $F$-subgroups $Q$ of $G$, whose Levi subgroup $L_Q$ is conjugate via $G(F)$ to $L_P$. As the opposite parabolic $\overline P$ of $P$ is the stabilizer of the opposite flag, every standard parabolic $F$-subgroup $P$ of $G$ is {\it self-associate}, i.e., conjugate to $\overline P$ by an element of $G(F)$. Consequently, $\{P\}$ consist only of $P$ itself, if $P$ is maximal parabolic.\\\\ A standard parabolic $F$-subgroup $P$ acts on its unipotent radical $N_P$ by the adjoint representation. The weights $\alpha$ of this action with respect to the torus $A_P$ are denoted $\Delta(P,A_P)$ and $\rho_P$ denotes the half-sum of these weights, counted with their multiplicity $m(\alpha)$ in $\n_P=Lie(N_P)$ (to ensure local-global compatibility of modulus characters). We will not distinguish between $\rho_P$ and its derivative, so we may also view $\rho_P$ as an element of $\check\a_P$. In particular, $\Delta^+_F:=\Delta(P_0,A_0)$ defines a choice of positive $F$-roots $\Delta_F$ of $G$ and we let $\Delta^0_F:=\{\alpha_1,...,\alpha_r\}$ be its basis of simple $F$-roots. Explicitly, the maximal $F$-split torus may be written as
$$A_0=\left\{\operatorname{diag}(t_1,\dots,t_r, id_{n-2r},c(t_1)^{-1},\dots,c(t_r)^{-1})\right\},$$
and let $\varepsilon_i\in X^\ast(A_0)$ be given by $\varepsilon_i(a)=t_i$ for $a\in A_0$. Then
$$
\Delta_F^+ =
\{\varepsilon_i\pm\varepsilon_j \ 1\leq i<j\leq r\}
\;\cup\;
\{\varepsilon_i \ 1\leq i\leq r\}
\;\cup\;
\{2\varepsilon_i\ 1\leq i\leq r\},
$$
with multiplicities in $\n_{0}$ given by
$$
m(\varepsilon_i\pm\varepsilon_j)=2,\qquad
m(\varepsilon_i)=2(n-2r)=2\ell_{P_0},\qquad
m(2\varepsilon_i)=1.
$$
Consequently, the abstract root system $\Delta_F$ is of reduced type $C_{r}$ if $\ell_{P_0}=0$ and of non-reduced type $BC_{r}$ otherwise. With respect to $\Delta^+_F$, we shall use the notation $\check\a_P^{+}$ (resp.\ $\overline{\check\a_P^{+}}$) for the open (resp.\ closed) positive Weyl chamber in $\check\a_P$.\\\\ The standard parabolic $F$-subgroups correspond one-to-one to subsets $\theta\subseteq\Delta^0_F$, characterized by the property that $\alpha\in\Delta^0_F\setminus \theta$ are precisely the simple $F$-roots, which do not vanish on $A_P$. If $P$ is maximal, i.e., $t_P=1$, there is hence precisely one such simple root $\alpha_k$, $1\leq k\leq r$, and we write $P=P_k$ with Levi subgroup $L_k \cong  \Res_{E/F}\GL_{k}\times G_\ell $ with $\ell=\ell_k=n-2k$, whose $F$-points may be viewed as

\begin{equation}\label{eq:LPk}
L_k(F)\cong
\left\{
\begin{pmatrix}
h & 0 & 0 \\
0 & u & 0 \\
0 & 0 & ({}^{c}h^t)^{-1}
\end{pmatrix}
:\;
h \in \GL_k(E),\; u \in G_{n-2k}(F)
\right\}.
\end{equation}
Therefore, 
$$\rho_{P_k}=(n-k)\,(\varepsilon_1+\cdots+\varepsilon_k).$$ 
Since the $F$-rational characters of $L_k$ are generated by $N_{E/F}\circ \det_{\GL_k/E}$ as a $\Z$-module, a generator of  $X^*(L_k)$ is given by $\gamma_k:=2(\varepsilon_1+\cdots+\varepsilon_k)$ and it follows that 
\begin{equation}\label{halfint}
s\cdot\gamma_k\in \tfrac12 X^*(L_k)\Rightarrow s\in\tfrac12\Z,
\end{equation} 
Going adelic and identifying an element $g\in L_k(\A_F)$ with a tuple $(h,u)\in\GL_k(\A_E)\times G_{n-2k}(\A_F)$ by \eqref{eq:LPk}, the above becomes $\|\det(h)\|^s_E=e^{\langle s\gamma_k,H_{P_k}(g)\rangle}$, for all $s\in\C$.\\\\
As all standard parabolic $F$-subgroups of $G$ are self-associate, all maximal parabolics $P_k$ (cf.\ \cite{moewal}, Lem.\ I.4.10, resp.\ their Remark, p.\ 162) serve as a potential candidate for a parabolic support of a residual automorphic representation of $G(\A_F)$.

\subsection{Real Lie groups and algebras}\label{sect:rlgrps}
The real locus $G_\infty:=\Res_{F/\Q}(G)(\R)$ of $G$ is given by 
\begin{equation}\label{eq:Ginfty}
G_\infty=\prod_{\vv\in S_\infty}G(F_{\vv})\cong \prod_{\vv\in S^{\sf in}_\R} U(p_\vv,q_\vv) \times \prod_{\vv\in S^{\sf sp}_\R} \GL_n(\R) \times \prod_{\vv\in S_\C} \GL_n(\C)\end{equation}
where for $\vv \in S^{\sf in}_\R$ we denoted by $(p_\vv,q_\vv)$, $p_\vv\geq q_\vv\geq 0$, the signature of the complex-conjugate hermitian form induced by $\<\cdot,\cdot\>$ on the $\C$-vector space $V_\vv := V\otimes_{E,\imath_\vv} \C$ and by $U(p_\vv,q_\vv)$ the attached real unitary group. It is a {\it real reductive} Lie group (see \cite{grob_book}, Rem.\ 2.1 for a precise positioning of this notion) whose component group satisfies $|\pi_0(G_\infty)|=2^{|S^{\sf sp}_\R|}$. The maximal compact subgroup $K_\infty$ of $G_\infty$, which is given by our choice of $K_{\A_F}$, fits into this by  
$$K_{\infty}= \prod_{\vv\in S_\infty}K_{\vv}\cong \prod_{\vv\in S^{\sf in}_\R}\left( U(p_\vv)\times U(q_\vv)\right) \times \prod_{\vv\in S^{\sf sp}_\R} O(n) \times \prod_{\vv\in S_\C}U(n).$$ Here, for any $m\geq 1$, we abbreviated as usual $U(m)=U(m,0)$ for the compact real unitary group of rank $m$. For any Lie subgroup $H$ of $G_\infty$, we let $K_{H}:=K^\circ_\infty\cap H$, so $K_G$ simply denotes the connected component of the identity matrix in $K_\infty$. It equals $K_\infty$ if and only if there is no split real place of $F$.\\\\
Lower case gothic letters denote the Lie algebra of the corresponding real Lie group (e.g., $\g_{\infty}=Lie(G_\infty)$, $\g_{\vv}=Lie(G(F_\vv))$, $\k_{H}=Lie(K_H)$, $\a_{P,\infty}=Lie(A_{P,\infty})$, etc.\ ...), and we add a subscript ``$\C$'' to denote its complexification (e.g., $\g_{\infty,\C}=\g_\infty\otimes_\R\C$). The real Lie algebra $\check\a_P$ (resp.\ $\a_P$) is viewed as being diagonally embedded into $\check\a_{P,\infty}$ (resp.\ $\a_{P,\infty}$) and we denote by $\mathcal Z(\g)$ the center of the universal enveloping algebra $\mathcal U(\g)$ of $\g_{\infty,\C}$.\\\\
Let $\h_\infty=\oplus_{\vv\in S_\infty}\h_\vv$ be a Cartan subalgebra of $\g_\infty$ that contains $\a_{0,\infty}$ (and hence all $\a_{P,\infty}$, $\a_{P,\vv}$ and $\a_P$). The choice of positivity on the set of $F$-roots of $G$ is extended to a choice of positivity on the set of absolute, complex roots $\Delta_\C:=\Delta(\g_{\infty,\C},\h_{\infty,\C})$. The half sum of the positive absolute roots is denoted $\rho=(\rho_\vv)_{\vv\in S_\infty}\in\check\h_\infty$. \\\\In this paper, we always let $\EE=\EE_\mu$ be a finite--dimensional irreducible algebraic representation of $G_\infty$ on a complex vector space which we assume to be given by a dominant integral highest weight $\mu=(\mu_\vv)_{\vv\in S_\infty}\in\check\h_\infty$. In order to avoid potential confusion, we remark that as $G_\infty$ is being viewed as a {\it real} Lie group, $\mu_\vv$ has in fact two coordinate vectors $\mu_{\iota_\vv}$ and $\mu_{\bar\iota_\vv}$ at a complex place $\vv\in S_\C\subseteq S_\infty$.

\subsection{Cohomological representations}
Let $H\subseteq G_\infty$ be the group of real points of an algebraic group over $\R$, such that $K_\infty\cap H$ is maximal among the compact subgroups of $H$. We say that a representation $\pi$ of $H$ is {\it cohomological}, if there is an irreducible finite-dimensionale algebraic representation $\cF$ of $H$ on a complex vector space, such that $H^*(\h, K^\circ_H,\pi\otimes \cF)\neq \{0\}$, i.e, that its relative Lie algebra cohomolgy does not vanish $H^*(\h, \k_H,\pi\otimes \cF)\neq \{0\}$: See \ \cite{bowa}, I.1.2 and I.5.2. We let $q_{min}(\pi)$ be the minimal degree, in which $\pi$ has non-vanishing $(\h, K^\circ_H)$-cohomology, or, equivalently $(\h, \k_H)$-)cohomology.\\\\In particular, a representation $\pi_\infty$ of $G_\infty$ is cohomological, if there is a highest weight module $\EE_\mu$ as above, such that $H^*(\g_{\infty},K_{G},\pi_\infty\otimes \EE_\mu)\neq \{0\}$. If $\pi_\infty\cong\otimes_{\vv\in S_\infty}\pi_\vv$ factors into representations of $G_\vv$ over the archimedean places, the K\"unneth-rule shows that the condition for $\pi_\infty$ of being cohomological is a purely local one, to be treated for each individual $\vv\in S_\infty$, cf.\ \ \cite{bowa}, I.5.1. Hence, in view of \eqref{eq:Ginfty}, we may consider the cases of real unitary groups $U(p,q)$ and of real and complex general linear groups $\GL_m(\R)$ and $\GL_m(\C)$ separately. The cohomological irreducible unitary representations, generally classified in \cite{vozu, voganunit}, are well and explicitly known for these families of groups (at least, if one admits the use of translation functors, cf.\ \cite{knappvogan}, Chp.\ VII or \cite{bowa}, Sect.\ VI.0): See \cite{enright} (for $\GL_m(\C)$), \cite{speh, speh81}  (for $\GL_m(\R)$), \cite{voganupq, salamribsu} (for $U(p,q)$) and alternatively and complementary \cite{grob-ragh}, Sect.\ 4 (for a very explicit study of the case of $\GL_{2m}(\R)$) and \cite{nairprasad}, Sect.\ 12 (for all the above mentioned groups, featuring results phrased in the language of Arthur-Johnson-packets). We now recall the tempered representations among those:

\subsubsection{Tempered cohomological representations of real unitary groups}
Let $U(p,q)$, $p\geq q\geq 0$, be a real unitary group and let $\sigma$ be an irreducible tempered representation of $U(p,q)$, which is cohomological with respect to an irreducible finite-dimensional (algebraic) representation $\cF$. Then $\sigma$ is isomorphic to one of the ${p+q \choose p}$ inequivalent discrete series representations, having the same infinitesimal character as $\cF^\vee$ (\cite{vozu}, p.\ 58, \cite{bowa}, Thm.\ I.5.3.(ii), \cite{knappbased}, Thm.\ 9.20, Thm.\ 12.21). Its cohomology is centered in the middle-degree
$$H^d(\u(p,q),U(p)\times U(q),\sigma\otimes\cF)\cong\left\{\begin{array}{ll}
 \C & \textrm{if $d=pq$} \\
 0 & \textrm{else}
\end{array}
\right.$$
See \cite{bowa}, II Thm.\ 5.4 or \cite{voganupq}, Thm.\ 3.2. The minimal degree, in which $\sigma$ has non-trivial cohomology, is hence, independent of $\sigma$ and $\cF$, given by $q_{min}(\sigma)=pq$.

\subsubsection{Tempered cohomological representations of general linear groups}\label{sect:tempGLcoh}
We refer to \cite{grob-ragh}, Sect.\ 5.5, for a concise and comprehensive summary, which is tailor-made for our purposes, and to further and more original references given therein. Let $\tau$ be a cohomological irreducible tempered representation of $\GL_m(\C)$, resp.\ $\GL_m(\R)$, $m\geq 1$. By the classification of the cohomological unitary dual of general linear groups, mentioned above, this is equivalent to saying that $\tau$ is a cohomological irreducible unitary {\it generic} representation of the respective group. In fact, the latter equivalence does not need the assumption that the respective coefficient module ist algebraic: An irreducible unitary representation $\tau$, which has non-zero $(\g\l_m(\C), U(m))$- (resp.\ $(\g\l_m(\R), SO(m))$-)cohomology with respect to some irreducible finite-dimensionale representation $\cF$ on a complex vector space, is generic if and only if it is tempered. Distinguishing the complex and the real case, we obtain the following:\\\\
{\it $F=\C$:} Then $\tau$ is fully induced from $m$ pairwise distinct characters $\chi_i: \C^*\ra U(1), z\mapsto z^{a_i}\overline{z}^{-a_i}$ with $a_i\in \tfrac12 \Z$,
$$\tau\cong {\rm Ind}_{B_m(\C)}^{\GL_m(\C)}[\chi_1\otimes...\otimes\chi_m]=\chi_1\times...\times\chi_m.$$
Here, $B_m$ denotes the standard Borel subgroup of $\GL_m/\Q$ of upper triangular matrices and induction is normalized by the square-root of its modulus character to ensure that unitarity is preserved. The $(\g\l_m(\C), U(m))$-cohomology of such a $\tau$ is hence concentrated in degrees $\tfrac{m(m-1)}{2}\leq d\leq \tfrac{m(m+1)}{2}$, cf.\ \cite{bowa}, Cor.\ III.5.2.(iii). \\\\
{\it $F=\R$:} Let $a=\lfloor\frac{m}{2}\rfloor$, $b=m-2a$ and let $P\supset B_m$ be the $\R$-parabolic subgroup of $\GL_m(\R)$ with Levi factor $M=\prod_{i=1}^{a} GL_{2}(\R)\times(\R^*)^b$ (were we understand that the factor $\R^*$ only appears, if $b=1$, i.e., if $m$ is odd). Then 
$$\tau\cong\textrm{Ind}^{\GL_m(\R)}_{P(\R)}[\delta_1\otimes...\otimes\delta_{a}\otimes\epsilon \ ]=\delta_1\times...\times\delta_{a}\times\epsilon.$$
Here, the $\delta_i$ are pairwise distinct discrete series representations of $\GL_2(\R)$ for $1\leq i\leq a$ and $\epsilon$ is a power of the sign character (if it appears, i.e., if $m$ is odd). The $(\g\l_m(\R), SO(m))$-cohomology of such a $\tau$ is hence concentrated, in dependence of the parity of $m$, in degrees $\tfrac{m^2-b}{4} \leq d\leq \tfrac{m(m+2)+b}{4}$, cf.\ \cite{bowa}, Prop.\ III.5.3.\\\\
In summary, the minimal degrees of cohomology, in which an irreducible tempered representation $\tau$ of $\GL_m(\C)$, resp.\ $\GL_m(\R)$, $m\geq 1$, may have non-trivial cohomology are again independent of $\tau$ and given by
$$q_{min}(\tau)=\left\{\begin{array}{ll}
 \tfrac{m(m-1)}{2} & \textrm{for $\GL_m(\C)$} \\\\
 \big\lfloor\tfrac{m}{2}\big\rfloor\big\lceil\tfrac{m}{2}\big\rceil & \textrm{for $\GL_m(\R)$} 
\end{array}
\right.$$ 
Let now $L_k$, be the Levi subgroup of a maximal parabolic $F$-subgroup $P_k$, $1\leq k\leq r$. As a simple consequence of the above short discussion and the K\"unneth rule for realtive Lie algebra cohomology, we obtain

\begin{prop}\label{prop:qminLk}
Let $\pi_\infty=\tau_\infty\hat\otimes\sigma_\infty$ be a cohomological irreducible tempered representation of $L_{k,\infty}=(\Res_{E/\Q}\GL_n)(\R)\times (\Res_{F/\Q})(U(V_{n-2k}))(\R)$. Then
$$q_{min}(\pi_\infty)=\sum_{\vv\in S^{\sf in}_\R}(p_\vv-k)(q_\vv-k) + |S^{\sf in}_\R| \cdot\tfrac{k(k-1)}{2}+ |S^{\sf sp}_\R| \cdot\left(\big\lfloor\tfrac{n-2k}{2}\big\rfloor\big\lceil\tfrac{n-2k}{2}\big\rceil+2\big\lfloor\tfrac{k}{2}\big\rfloor\big\lceil\tfrac{k}{2}\big\rceil\right) + |S_\C| \cdot\tfrac{n^2-4kn-n+6k^2}{2}$$
\end{prop}

\subsection{Automorphic forms for $G_n(\A_F)$ and $\GL_n(\A_E)$}
Let $\mathcal A(G)$ (resp.\ $\mathcal A^\infty(G)$) be the space of all automorphic forms (resp.\ smooth-automorphic forms) on $G(\A_F)=G_n(\A_F)$, cf.\ \cite[Sect.\ 4.5]{bojac} and \cite[Sect.\ 11.5 and Def.\ 11.10]{grob_book}, respectively. It follows from the general theory, that $\mathcal A(G)\cap L^2(G(F)\backslash G(\A_F)=\mathcal A(G)\cap L^2_{dis}(G(F)\backslash G(\A_F))$, since the continuous part of the $L^2$-spectrum does not contain non-zero $K_\infty$- and $\mathcal Z(\g)$-finite elements. Hence, the (irreducible) square-integrable automorphic representations may be identified with the (irreducible) $(\g_\infty,K_\infty, G(\A_F))$-submodules of the space of globally smooth, $K_\infty$- and $\mathcal Z(\g)$-finite elements in the discrete part of the $L^2$-spectrum, i.e., with (irreducible) $(\g_\infty,K_\infty, G(\A_F))$-submodules of $L^2_{dis}(G(F)\backslash G(\A_F))^{\infty_\A}_{(K_\infty,\mathcal Z(\g))}$. Here we used the notation of \cite{grob_book}, Thm.\ 13.12 and Cor.\ 13.20, to which we refer for a detailed account, while we recall that $A_G$ is trivial, so $G(F)\backslash G(\A_F)$ is of finite volume.\\\\ 
It will be convenient -- and in view of \cite{grob_book}, Lem.\ 15.10 justified -- to allow ourselves not to distinguish between an irreducible (unitary) $G(\A_F)$-representation on a subspace $\mathcal H$ of the Hilbert space $L^2_{dis}(G(F)\backslash G(\A_F))$, the irreducible smooth $G(\A_F)$-representation on the limit-Fr\'echet-space of its globally smooth vectors $\mathcal H^{\infty_\A}$ and its underlying irreducible $(\g_\infty,K_\infty, G(\A_F))$-module $\mathcal H^{\infty_\A}_{(K_\infty)}$ of $K_\infty$-finite vectors. In other words, we will sometimes make use of the freedom not to distinguish between an irreducible square-integrable automorphic representation of $G(\A_F)$, its smooth $LF$-space completion or its (non-smooth) Hilbert space completion in the $L^2$-spectrum. See \cite{grob_book}, Sect.\ 10 and Thm.\ 12.10, {\it ibidem}, and \cite{grob_zun}, Thm.\ 3.7. We only remark that the middle ``$LF$-perspective'' has the advantage over the classical ``$K_\infty$-finite perspective'' of not being dependent on the specific choice of a maximal compact subgroup, but being intrinsic to the underlying group scheme $G/F$ and providing continuous, smooth representations of $G(\A_F)$ (and not merely a $(\g_\infty,K_\infty, G(\A_F))$-module); while it has the advantage over the ``Hilbert space-perspective'' that its elements are proper functions (and not merely classes of functions), which is why, for instance, the evaluation mapping $ev_g$, for $g\in G(\A_F)$, is defined on $\mathcal A^\infty(G)$ (but clearly not on subspaces of $L^2_{dis}(G(F)\backslash G(\A_F))$).
If not otherwise specified, we shall from now on assume that square-integrable automorphic representations are assumed to be irreducible.\\\\
Irreducible square-integrable automorphic representations of $G(\A_F)$ are in connection with the near equivalence classes of irreducible automorphic representations of $\GL_n(\A_E)$, i.e., after passing to a system of representatives of the latter, with the isobaric automorphic representations of $\GL_n(\A_E)$ by base change. To this end, we recall the following

\begin{thm}\label{thm:lift}
Let $\pi$ be a square-integrable automorphic representation of $G_n(\A_F)$. Then, $\pi$ admits a global {\it base change} to an automorphic representation $BC(\pi) =: \Pi$ of $\GL_n(\A_E)$. This representation $\Pi$ is an isobaric sum $\Pi=\Pi_1\boxplus...\boxplus\Pi_l$ of conjugate self-dual (with respect to the non-trivial Galois automorphism $c$ of $E/F$) square-integrable automorphic representations $\Pi_i$ of $\GL_{n_i}(\A_E)$, $n=n_1+....+n_l$ some partition of $n$ . It is uniquely characterized by the property that for every achimedean place and at every non-archimedean place $\vv\in S$, where $\pi_\vv$ is unramified, $\Pi_\vv:=\otimes_{{\sf w} | {\sf v}}\Pi_{\sf w}$ is obtained from $\pi_\vv$ by local base change, cf.\ \cite[Prop.\ 1.3.3]{KMSW}. If $\pi_\infty$ is cohomological, then $\Pi_\infty=\otimes_{\ww\in \Sigma_\infty}\Pi_\ww$ is a cohomological representation of $\Res_{E/\Q}(\GL_n)(\R)$. 
\end{thm}
\begin{proof}
Existence of base change in the above form is due to the work of many people, cf.\ \cite{lab, harris-labesse, kim-krish04, kim-krish05, morel, shin, ckpssh}, to name a few, culminating in \cite{mok} (for quasisplit unitary groups) and \cite{KMSW, zou} (for the general case: see Thm$.^*$ 1.7.1 and Thm.\ 5.0.5 in the former and Thm.\ 2.1 and Prop.\ 4.1 in the latter reference)\footnote{A certain {\it caveat} is to be made: The references \cite{mok, KMSW, zou} depend in the work of Arthur, cf.\ \cite{arthur}, which itself still depends on a proof of the twisted weighted fundamental lemma (the ``local intertwining relation'' being settled in the quasi-split case by the recent fundamental work \cite{AGIKMS}.)}. It is well-known that the local archimedean base change of a cohomological representation is again cohomological, see \cite{lab} \S 5.1, \cite{clozelihes}, Sect.\ 3.3, (for tempered representations of real unitary groups), \cite{johnson}, or, perhaps most tailor-made, \cite{nairprasad}, Sect.\ 12.
\end{proof}

\noindent Recall that by well-known local results of Bernstein, Jacquet (non-archimedean) and Vogan (archimedean), $\Pi\cong{\rm Ind}_{P_{n_1,...,n_l}(\A_E)}^{\GL_n(\A_E)}[\Pi_1\otimes...\otimes\Pi_l]$ is fully induced from the conjugate self-dual square-integrable automorphic representations $\Pi_i$ (again, global parabolic induction is assumed to be normalized by the Hecke character given by $\rho_{P_{n_1,...n_l}}$, in order to preserve unitarity), cf.\ \cite{grob_book}, \S 4.3.5 for detailled references and a general exposition.\\\\ 
We shall define the  {\it Asai $L$-functions} (of sign $\pm$) attached to a cuspidal automorphic representation $\tau\cong \otimes'_\ww \tau_\ww$ of $\GL_n(\A_E)$ as the completed automorphic $L$-function 
$$L(s,\tau,{\rm As}^{\pm}):=\prod_{\ww\in \Sigma} L(s,\tau_\ww,{\rm As}^{\pm})$$
given by the product over all local Langlands $L$-functions $L(s,\tau_\ww,{\rm As}^{\pm})$ attached to the according Asai-representation ${\rm As}^{\pm}$ of the $L$-group of $\GL_n/E$. We refer to \cite{ggp}, pp.\ 26 and 82--83 or \cite{grbacshahidi}, Sect.\ 2.A, for details and to \cite{henniartAsai}, Thm.\ 1.5, for the compatibility of this definition with the other, partly older definitions in the literature. Thm.\ \ref{thm:lift} is now complemented by

\begin{thm}\label{thm:qtcusp}
Let $\pi$ be a cohomological, square-integrable automorphic representation of $G_n(\A_F)$ and write $\Pi=\Pi_1\boxplus...\boxplus\Pi_l$ for its base change to $\GL_n(\A_E)$. Then the following assertions are equivalent:
\begin{enumerate}
\item The summands $\Pi_i$ are cuspidal for all $1\leq i\leq l$.
\item $\Pi$ is tempered at all archimedean places and quasi-tempered at all non-archimedean places $\ww\in\Sigma$. Here, the latter means that for each non-archimedean $\ww\in\Sigma$, there is a partition $n=m_1+...+m_u$, discrete series representations $\delta_{i}$ of $\GL_{m_i}(E_\ww)$, and real numbers $a_{i}$ with $|a_{i}|<\tfrac12$, such that $\Pi_\ww\cong\delta_{1} |.|_\ww^{a_{1}}\times...\times\delta_{u} |.|_\ww^{a_{u}}$. 
\end{enumerate}
If one of the above equivalent assertions holds, then $\pi$ is necessarily cuspidal, $\pi_\infty$ is tempered and the isobaric summands $\Pi_i$ are all distinct, $\Pi_i\ncong\Pi_j$, for $i\neq j$. Moreover, for each $1\leq i\leq l$, the partial Asai $L$-function $L^T(s,\Pi_i,{\rm As}^{(-1)^{n-1}})$ has a pole at $s=1$ for a sufficiently large, but finite set of places $\Sigma_\infty\subseteq T\subset \Sigma$.
\end{thm}
\begin{proof}
This is well-known. Lacking a good reference, we sketch the argument. Suppose that (1) holds. Since $\Pi\cong{\rm Ind}_{P_{n_1,...,n_l}(\A_E)}^{\GL_n(\A_E)}[\Pi_1\otimes...\otimes\Pi_l]$ is fully induced and since cuspidal automorphic representations of $\GL_m(\A_E)$, $m\geq 1$, are all globally generic (cf. \cite{shal} corollary on p. 190), $\Pi_\ww$ is quasi-tempered at all non-archimedean places $\ww\in\Sigma$ by the classification of the generic unitary dual of the general linear group over a non-archimedean local field, see \cite{kudla}, \S 4.3 - 4.4, for a general exposition and further references. At an archimedean place $\Pi_\ww$ is tempered, since it is cohomological, cf.\ Thm.\ \ref{thm:lift}, unitary and (being induced from generic representations itself) generic, see \S \ref{sect:tempGLcoh} above. This shows (2).\\\\ We continue to assume that (1) holds. Then, $\Pi$ serves as a generic, elliptic global Arthur-parameter $\phi$ in the sense of \cite{KMSW}, \S 1.3.4 or \cite{zou}, \S 2.1. This has two consequences: On the one hand, its localization $\phi_\vv$ (cf.\ \cite{KMSW}, Prop.\ 1.3.3) at any place $\vv\in S$ , where $\Pi_\vv$ is tempered, is bounded, because so is the local Langlands-parameter attached to $\Pi_{\vv}$ by the Local Langlands Correspondence, \cite{HT, henniart}. Hence, item (5) of Thm.\ 1.6.1 of \cite{KMSW} implies that each square-integrable automorphic representation of $G(\A_{F})$ attached to $\phi$ by \cite{KMSW}, Thm.\ 5.0.5 is tempered at every such place $\vv\in S$. See also \cite{zou}, Thm.\ 2.5.1.(2). In particular, 
\begin{equation}\label{eq:temp}
\Pi_\vv \ \textrm{tempered}\Rightarrow \pi_\vv \ \textrm{tempered.}
\end{equation}
Since we just observed that $\Pi_\infty$ is tempered, we obtain that $\pi_\infty$ is tempered, which in turn implies that $\pi$ is necessarily cuspidal, \cite{wallach}, Thm.\ 4.3.\\\\ 
On the other hand, it follows directly from the nature of the global parameter $\phi$, that the partial Asai $L$-functions $L^T(s,\Pi_i,{\rm As}^{(-1)^{n-1}})$ have a pole at $s=1$ for a sufficiently large, but finite set of places $\Sigma_\infty\subseteq T\subset \Sigma$ for each $1\leq i\leq l$, cf.\ \cite{zou}, Thm.\ 2.1 and Sect.\ 2.1 {\it ibidem}. 
So (1) implies also the other claims made in the statement of the theorem. \\\\
To complete the proof, let us finally assume that (2) holds. By the classification of the discrete spectrum of $\GL_m(\A_E)$ by M\oe glin-Waldspurger, \cite{MW}, every square-integrable automorphic summand $\Pi_i$ of $\Pi$ is itself an isobaric sum of the form 
$$\Pi_i=\Theta_i \|\cdot \|_E^{\tfrac{b_i-1}{2}}\boxplus \Theta_i \|\cdot \|_E^{\tfrac{b_i-3}{2}} \boxplus ...\boxplus \Theta_i \|\cdot \|_E^{-\tfrac{b_i-3}{2}} \boxplus\Theta_i \|\cdot \|_E^{-\tfrac{b_i-1}{2}},$$ 
for $\Theta_i$ unitary cuspidal automorphic and $b_i\geq 1$ an integer. Recalling the already used fact that the local components of $\Theta_i$ are quasi-tempered, because they are generic and unitary, assumption (2) (together with the uniqueness of the square-integrable local Langlands datum) forces $b_i=1$, i.e., that each summand $\Pi_i$ of $\Pi$ is cuspidal, $\Pi_i=\Theta_i$, as claimed by (1).
 \end{proof}

\begin{rmk}
If $E/F$ is a CM-extension, then it was proved in \cite{GHL}, Prop.\ 3.4, that for a cohomological cuspidal automorphic representation $\pi$ of $G(\A_F)$, the equivalent conditions of Thm.\ \ref{thm:qtcusp} are furthermore equivalent to
\begin{enumerate}
\item[(4)] $\pi_\vv$ is tempered at every place $\vv\in S$. 
\item[(5)] $\pi$ contributes to cute\footnote{This word being an acronym for {\it cu}spidal {\it te}mpered, see \cite{GHL}.} coherent cohomology of $G(F)\backslash G(\A_F)/K_\infty$.
\end{enumerate}
A key-ingredient is \cite{car}, Thm.\ 1.2, which established the Ramanujan conjecture for the (cohomological twists of the) isobaric summands $\Pi_i$ of $\Pi=BC(\pi)$. Indeed, whenever one works over a number field $E$, for which the Ramanujan conjecture is known for conjugate self-dual cuspidal automorphic representations of $\GL_n(\A_E)$, then our observation \eqref{eq:temp} made in course of the proof of Thm.\ \ref{thm:qtcusp} will imply that $\pi$ is tempered at every place $\vv\in S$, once its base change $\Pi$ has only cuspidal automorphic summands.
\end{rmk}

\section{Poles of Eisenstein series}
\subsection{Parabolic and cuspidal supports}\label{sect:pc}
Let $\J$ be the ideal of $\mathcal Z(\g)$, which annihilates the contragredient representation $\EE_\mu^\vee$ of $\EE_\mu$. It is of finite codimension in $\mathcal Z(\g)$ and we denote by $\mathcal A^\infty_\mathcal J(G)$ the attached space of smooth-automorphic forms, which are annihilated by some power of $\mathcal J$. Unlike  $\mathcal A^\infty(G)$, it carries an $LF$-space structure in a natural way and equipped with right-translation of functions, this turns $\mathcal A^\infty_\mathcal J(G)$ into a smooth $G(\A)$-representation, cf.\ \cite{grob_book}, Thm.\ 11.17 or \cite{grob_zun}, Prop.\ 2.15. For an associate class  $\{P\}$ of parabolic $F$-subgroups, we denote by $\mathcal A^\infty_{\J,\{P\}}(G)$ its $G(\A)$-subrepresentation on the $LF$-space of all $\varphi\in\mathcal{A}^\infty_\J(G)$, which are negligible along every parabolic $F$-subgroup $Q\notin\{P\}$, i.e., for which the respective constant term $\varphi_Q$ satisfies
$$\lambda_{Q,g,\phi}(\varphi):=\int_{L_Q(F)\backslash L_Q(\A_F)^{(1)}} \varphi_Q(lg) \,\overline{\phi(l)} \,dl =0,$$
for all $g\in G(\A_F)$ and all cuspidal smooth-automorphic forms $\phi\in \mathcal A_{cusp}^\infty(L_Q)$. Then, $\mathcal{A}^\infty_\J(G)$ decomposes as a $G(\A)$-representation, hence, also topologically as
\begin{equation}\label{eq:parabsuppdec}
\mathcal{A}^\infty_\J(G)\cong\bigtoplus_{\{P\}}\mathcal A^\infty_{\J,\{P\}}(G),
\end{equation}
cf. \cite{grob_book}, Thm.\ 14.17. (Here and in what follows the use of the symbol $\bigtoplus$ has proved to be a convenient way to indicate that the respective isomorphism is also an isomorphism of underlying locally convex vector spaces.)\\\\ In \cite{grob_book}, \S 15 and \cite{grob_zun}, \S 4.6, the various summands $\mathcal A^\infty_{\mathcal J,\{P\}}(G)$ were decomposed even further. To this end, we recall the notion of an {\it associate class} $\varphi_{P}$ of cuspidal smooth-automorphic representations of the Levi subgroups of the elements in the class $\{P\}$: These associate classes are parameterized by pairs of the form $(\pi,\Lambda)$, where
\begin{enumerate}
\item $\pi$ is a unitary cuspidal (smooth-)automorphic representation of $L(\A_F)$, whose central character vanishes on the group $A^\R_P$
\item $\Lambda:A^\R_P\rightarrow\C^*$, a complementary Lie group character such that
\item the Weyl group orbit of the infinitesimal character $\chi_{\pi}$ of $\pi_\infty$ and the derivative $\lambda_0:=d\Lambda\in\check\a_{P,\C}$ of $\Lambda$ annihilate the ideal $\mathcal J$. We refer to \cite{schwfr}, 1.2, for the original and \cite{grob_book}, \S 15.2, for a more tailor-made and detailed source. Here we only remark that this condition implies that $-(\mu+\rho)$ lies in the Weyl group orbit of $\chi_{\pi}+\lambda_0$ 
\end{enumerate}
Hence, as $A^\R_P$ is a direct factor of $L(\A_F)$, we may and will view each associate class $\varphi_{P}$ equally well as being represented by a (potentially non-unitary) cuspidal (smooth-)automorphic representation 
$$\pi_{\lambda_0}:=\pi\otimes e^{\langle \lambda_0,H_{P}(\cdot)\rangle}$$ of $L(\A_F)$ subject to the above conditions. We record the following, technically important result:

\begin{lem}\label{lem:varphis}
Let $\varphi_{P}$ be an associate class of cuspidal automorphic representations of $L(\A_F)$, represented by $\pi_{\lambda_0}=\pi\otimes e^{\langle \lambda_0,H_{P}(\cdot)\rangle}$. Write $\pi=\tau_1\hat\otimes...\hat\otimes\tau_{t_P}\hat\otimes\sigma$ according to the decomposition 
$L(\A_F)\cong \prod_{i=1}^{t_P}\GL_{k_i}(\A_E)\times U(V_{\ell_P})(\A_F)$, cf.\ \eqref{eq:LP} Then, each $\tau_{i,\infty}$ is tempered and $\sigma_\infty$ is cohomological.
\end{lem}
\begin{proof}
Recall our choice of the ideal $\mathcal J$ of $\mathcal Z(\g)$, which annihilates the contragredient representation $\EE_\mu^\vee$ of $\EE_\mu$. Condition (3) from above hence ensures that the archimedean component of $\pi_{\lambda_0}\otimes e^{\langle \rho_P,H_{P}(\cdot)\rangle}$ has the same infinitesimal character as the contragredient of an irreducible summand $\mathcal F$ of the semisimple $L_\infty$-representation $H^*(\n_{P,\infty},\EE_\mu)$. The latter is algebraic, because so is the action of $L_\infty$ on $\EE_\mu$ and on $\n_{P,\infty}$.\\\\ 
To analyze this further, let $M_\infty=\prod_{\vv\in S_\infty} M_\vv$ be the subgroup of $L_\infty$, defined by $M_\vv:=\bigcap \ker |\xi_\vv|$, the intersection running over all continuous characters $\xi_\vv:L(F_\vv)\rightarrow\C^*$. It is a reductive Lie group with compact center and admits a direct complement in $L_\infty=M_\infty\times A_\infty$, with $A_\infty\supseteq A_P^\R$. Therefore, the restriction of $(\pi_{\lambda_0}\otimes e^{\langle \rho_P,H_{P}(\cdot)\rangle})_\infty$ to the connected component $M_\infty^\circ$ of $M_\infty$ (which is of finite index in the latter) equals $\pi_\infty |_{M_\infty^\circ}$, which is furthermore isomorphic to the finite direct sum of irreducible unitary representations
$$\pi_\infty |_{M_\infty^\circ}\cong \pi_1\oplus...\oplus\pi_y$$
of the connected Lie group $M_\infty^\circ$. Up to a finite cover, $M^\circ_\infty$ is the product of its derived semisimple subgroup 
$$\mathcal D(M^\circ_\infty)\cong \prod_{i=1}^{t_P} {\rm SL}_{k_i,\infty}\times SU(V_{\ell_P})_\infty,$$ 
and a compact central torus $U(1)^{(t_P+1)|S^{\sf in}_\R|+(2t_P+1)|S_\C|}$. Any irreducible unitary Harish-Chandra module for $M^\circ_\infty$ therefore decomposes, after passing to this finite cover, as the product of an irreducible unitary Harish-Chandra module for the derived group $\mathcal D(M^\circ_\infty)$ and a unitary character of the compact center. Analogously, $\cF|_{M_\infty^\circ}$ breaks as the finite direct sum of pairwise isomorphic irreducible finite-dimensional representations of $M_\infty^{\circ}$, which we shall all, as they are all pairwise isomorphic, simply denote by $\cF^{\circ}$. The the summands $\pi_1$, ..., $\pi_y$ all have the same infinitesimal character as $(\cF^\circ)^\vee$, which follows from condition (3) from above: See \cite{bowa}, Rem.\ 1 on p.\ 64. As a consequence, \cite{salamribsu}, Thm.\ 1.8 applies to the connected derived semisimple factor $\mathcal D(M^\circ_\infty)$ and shows that $\pi_1|_{\mathcal D(M^\circ_\infty)}$, ..., $\pi_y|_{\mathcal D(M^\circ_\infty)}$ all have non-zero $(\mathcal D(\m_\infty), K_{\mathcal D(M^\circ_\infty)})$-cohomology with respect to $\cF^{\circ}|_{\mathcal D(M^\circ_\infty)}$. Also, the just mentioned match of the infinitesimal characters shows that the compact connected central torus $U(1)^{(t_P+1)|S^{\sf in}_\R|+(2t_P+1)|S_\C|}$ acts trivially on $\pi_j\otimes \cF^{\circ}$, $1\leq j\leq y$, whence $\pi_\infty$ has non-trivial $(\m_\infty, K_{M^\circ_\infty})$-cohomology with respect to $\cF^\circ$.\\\\ 
Since $\prod_{i=1}^{t_P}{\rm SL}_{k_i,\infty}=\prod_{i=1}^{t_P}\prod_{\ww\in\Sigma_\infty} {\rm SL}_{k_i}(E_\ww)$ applying the K\"unneth-rule yields that for each $\ww\in \Sigma_\infty$ and each $1\leq i\leq t_P$, $\tau_{i,\ww}|_{{\rm SL}_{k_i}(E_\ww)}$ is the finite direct sum of irreducible unitary cohomological representations of ${\rm SL}_{k_i}(E_\ww)$. At least one of them, call it $\tau'_{i,\ww}$, is generic, which follows from genericity of $\tau_{i,\ww}$, and such a $\tau'_{i,\ww}$ is therefore tempered, see \S \ref{sect:tempGLcoh} (translated from ${\rm GL}_{k_i}(E_\ww)$ to the case of ${\rm SL}_{k_i}(E_\ww)$, which only needs a short argument, if $E_\ww\cong\R$, cf.\ \cite{speh81}, Sect.\ 1.3--1.5, for the latter). Inducing $\tau'_{i,\ww}$ from ${\rm SL}_{k_i}(E_\ww)$ to its finite cover, given by the maximal semisimple subgroup of $M_\vv\cap {\rm GL}_{k_i}(E_\ww)$ (which is ${\rm SL}^\pm_{k_i}(\R)$, if $\ww$ lies above a split real place $\vv\in S^{\sf sp}_\R$ and ${\rm SL}_{k_i}(\C)$ else), preserves temperedness, so, recalling that the compact central torus of $M^\circ_\vv$ as well as $A_\vv$ act by a (unitary) character on $\tau_{i,\ww}$, we are led to conclude that $\tau_{i,\ww}$ is tempered for every $\ww\in \Sigma_\infty$ and every $1\leq i\leq t_P$, whence so are all the representations $\tau_{i,\infty}$, $1\leq i\leq t_P$. This shows the first claim.\\\\
For the second claim, observe that the irreducible $U(V_{\ell_P})_\infty$-factor of $\pi_\infty$, which is just $\sigma_\infty$ by definition, is the same as of $(\pi_{\lambda_0}\otimes e^{\langle \rho_P,H_{P}(\cdot)\rangle})_\infty$, because $A_P$ has trivial intersection with $U(V_{\ell_P})$. But since $\pi_\infty$ has non-zero $(\m_\infty, K_{M^\circ_\infty})$-cohomology with respect to $\cF^{\circ}$, the $L_\infty$-representation $(\pi_{\lambda_0}\otimes e^{\langle \rho_P,H_{P}(\cdot)\rangle})_\infty$ is cohomological with respect to $\cF$, which follows from condition (3) from above, see again \cite{bowa}, Rem.\ 1 on p.\ 64. Therefore, using the K\"unneth-rule shows that the $U(V_{\ell_P})_\infty$-representation $\sigma_\infty$ is cohomological.
\end{proof}

Attached to $\pi$ let $I^G_{P}(\pi)$ be the space of all smooth, left $L(F)N(\A)A^\R_P$-invariant functions $f: G(\A)\ra\C$, such that for every $g\in G(\A)$ the function $l\mapsto f(lg)$ on $L(\A)$ is contained in the $\pi$-isotypic component $\pi^{m(\pi)}$ of the $LF$-space $\mathcal{A}^\infty_{cusp}(L)$. For a function $f\in I^G_{P}(\pi)$, $\lambda\in\check\a_{P,\C}$ and $g\in G(\A)$ an {\it Eisenstein series} may be formally defined as
\begin{equation}\label{eq:eis}
E_{P}(f,\lambda)(g):=\sum_{\gamma\in P(F)\backslash G(F)}
f(\gamma g) e^{\<\lambda+\rho_P,H_{P}(\gamma g)\>}.
\end{equation}
If $f$ is $K_\infty$-finite, the so-defined Eisenstein series is known to converge absolutely and uniformly on compact subsets of $G(\A)\times \{\lambda\in\check\a_{P,\C}| \Re e(\lambda)\in\rho_P+\check\a^{+}_P\}$, cf.\ \cite{moewal}, Prop.\ II.1.5. For such $\lambda$, $E_{P}(f,\lambda)$ is a smooth-automorphic form. In turn, given $g\in G(\A_F)$, the map $\lambda\mapsto E_{P}(f,\lambda)(g)$ can be continued to a meromorphic function on all of $\check\a_{P,\C}$, whose singularities (i.e., poles) lie along affine hyperplanes of the form $R_{\alpha, C}:=\{\xi\in\check\a_{P,\C}| (\xi,\alpha)=C\}$ for some constant $C$ and some root $\alpha\in\Delta(P,A_P)$, called ``root-hyperplanes'' (cf.\ \cite{moewal}, Thm.\ IV.1.8, Prop. IV.1.11 (a); \cite{langlandsLNM}, pp.\ 170--171; or, \cite{BL}, Thm.\ 2.3). This entails the assertion that for each $\xi_0\in \check\a_{P,\C}$, there is a minimal integer $k\geq 0$, such that the function 
$$q_{\xi_0}(\lambda):=\prod_{\alpha\in\Delta(P,A_P)}\langle \lambda-{\xi_0},\check\alpha\rangle^k,$$
is a non-zero, holomorphic function in $\lambda\in \check\a_{P,\C}$, with the property that the assignment $ \check\a_{P,\C}\ra\C$, which sends $\lambda \mapsto q_{\xi_0}(\lambda) E_P(f,\lambda)(g)$ is holomorphic in a small neighbourhood of ${\xi_0}$ for all $K_\infty$-finite $f\in I_P^G(\pi)$ and $g\in G(\A)$. As it is clear from the definition of $I_P^G(\pi)$, evaluation of functions at the identity element $g=id\in G(\A_F)$ identifies each $\phi\in {\rm Ind}_{P(\A_F)}^{G(\A_F)}[\pi^{m(\pi)}\otimes e^{\langle \lambda,H_{P}(\cdot)\rangle}]$ with a unique summand of the form $f\cdot e^{\<\lambda+\rho_P,H_{P}(\cdot)\>}$ as in \eqref{eq:eis}.\\\\
Let ${S}(\check\a_{P,\C})$ be the symmetric algebra of $\check\a_{P,\C}$, viewed as the space of differential operators $\partial$ with constant coefficients on $\check\a_{P,\C} $. Then, at $\lambda_0=d\Lambda$ as above, the formal assignment defined by
\begin{equation}\label{eq:Eismap}
{\rm Eis_{\pi,\lambda_0}}(f\otimes \partial):=\partial(q_{\lambda_0}(\lambda) E_P(f,\lambda))|_{\lambda=\lambda_0}
\end{equation}
is well-defined on the $K_\infty$-finite elements $f\in I^G_P(\pi)_{(K_\infty)}$ and $\partial\in{S}(\check\a_{P,\C})$ and we let 
\begin{equation}\label{eq:defcuspsup}
\mathcal A^\infty_{\mathcal J,\{P\},\varphi_P}(G):={\rm Cl}_{\mathcal A^\infty_{\mathcal J}(G)}({\rm Eis_{\pi,\lambda_0}}(I^G_P(\pi)_{(K_\infty)}\otimes {S}(\check\a_{P,\C})))
\end{equation}
denote the topological closure of its image in $\mathcal{A}^\infty_{\J}(G)$. Its definition is independent of the choice of the representatives $P$ and $(\pi,\Lambda)$, thanks to the functional equations satisfied by the Eisenstein series considered, cf.\ \cite{moewal}, Thm.\ IV.1.10 and  \cite{schwfr} 1.2-1.4. Moreover, $\mathcal A^\infty_{\mathcal J,\{P\},\varphi_P}(G)$ is a smooth-automorphic subrepresentation, lying inside the $G(\A_F)$-representation on the $LF$-space $\mathcal A^\infty_{\mathcal J,\{P\}}(G)$. It is the main result of \cite{grob_zun} that one has in fact a direct sum decomposition into closed $G(\A)$-subrepresentations
\begin{equation}\label{eq:cuspsuppdec}
	\mathcal{A}^\infty_{\J,\left\{P\right\}}(G)=\bigtoplus_{\varphi_P}\mathcal A^\infty_{\mathcal J,\{P\},\varphi_P}(G), 
\end{equation}
the topological direct sum ranging over all associate classes of cuspidal smooth-automorphic subrepresentations of $L(\A_F)$ as above. See also \cite{grob_book}, Thm.\ 15.21.\\\\ 
By their very definition, see \eqref{eq:defcuspsup} above, in order to describe the $G(\A)$-representations $\mathcal A^\infty_{\mathcal J,\{P\},\varphi_P}(G)$, one has to describe the (limits of) Cauchy filters of partial derivatives of holomorphic residues of Eisenstein series $E_{P}(f,\lambda)$ as above. In particular, it is decisive to know, whether for a given associate class $\varphi_P$ there are smooth $K_\infty$-finite sections $f\in I^G_P(\pi)$ such that the attached Eisenstein series $E_{P}(f,\lambda)$ have a pole at the point $\lambda_0\in\check\a_{P,\C}$. To this end, one has to study the poles of so-called global standard intertwining operators, as defined in \cite{moewal}, II.1.6. Indeed, combining \cite{moewal}, Lem.\ I.4.10 with Prop.II.1.7 and Sect,\ IV.4.1, and recalling that all standard parabolic $F$-subgroups $P$ of $G$ are self-associate, one sees that the poles of an Eisenstein series $E_{P}(f,\lambda)$ are contained in the union of all poles of the standard global intertwining operators appearing in its constant term along $P$ itself.

\subsection{Intertwining operators}\label{sect:intops}
We shall render the above now more precise in the case of maximal parabolic $F$-subgroups. Therefore, let now $P=P_k$, $1\leq k\leq r$, be a maximal parabolic $F$-subgroup of $G$ with Levi decomposition $P_k=L_kN_k$, cf.\ \S \ref{sect:paras}, and let $\varphi_{P_k}$ be an associate class of cuspidal automorphic representations, represented by a cuspidal automorphic representation $\pi_{\lambda_0}=\pi\otimes e^{\langle \lambda_0,H_{P}(\cdot)\rangle}$ of $L_k(\A_F)$. We write $\lambda_0=s_0\gamma_k$ with respect to our fixed generator $\gamma_k=2(\varepsilon_1+\cdots+\varepsilon_k)$ of $X^*(L_k)$. Moreover, since $L_k \cong  \Res_{E/F}\GL_{k}\times G_{n-2k}$, cf.\ \eqref{eq:LPk}, we may write, $\pi=\tau\hat\otimes\sigma$, where $\tau$ is a unitary cuspidal automorphic representation of $\GL_k(\A_E)$ and $\sigma$ is a unitary cuspidal automorphic representation of $G_{n-2k}(\A_F)$. The quotient group $W_k:=N_{G(F)}(A_{P_k}(F))/L_k(F)$ has two elements and we choose and fix a representative $w_k\in G(F)$ for its unique non-trivial element. We let $w_k(\pi)$ be the representation $w_k(\pi)(g)=\pi(w_k gw_k^{-1})$, which, due to multiplicity one for the cuspidal spectrum of $\GL_k(\A_E)$, turns out to be equal to $({}^c\tau^\vee)\hat\otimes\sigma$, where ${}^c\tau^\vee$ denotes the conjugate dual representation of $\tau$ (with respect to the non-trivial Galois automorphism $c$ of $E/F$). (It is hence independent of the very choice of $w_k$.)\\\\ For $s\in\C$ with $\Re e(s)\gg 0$  we consider the usual global intertwining operator
$$M(s,\pi): {\rm Ind}_{P_k(\A_F)}^{G(\A_F)}[\pi\otimes e^{\langle s\gamma_k,H_{P_k}(\cdot)\rangle}]\longrightarrow {\rm Ind}_{P_k(\A_F)}^{G(\A_F)}[w_k(\pi)\otimes e^{\langle -s\gamma_k,H_{P_k}(\cdot)\rangle}]$$
$$\phi \mapsto M(s,\pi)\phi: g\mapsto \int_{(N(\A_F)\cap w_k N(\A_F)w_k^{-1})\backslash N(\A_F)} \phi(w_k^{-1}ng) \ dn,$$
cf.\ \cite{moewal}, II.1.6, already shortly mentioned at the end of \S \ref{sect:pc} above. It is well-known that $M(s,\pi)$ extends to a meromorphic operator in $s$ to all of $\C$ (\cite{moewal}, Thm.\ IV.1.8 or Sect.\ IV.3 therein) and that -- since it is the only non-trivial intertwining operator appearing in the constant term of $E_{P_k}(f,s\gamma_k)$ along $P_k$ -- the poles of the Eisenstein series $E_{P_k}(f,s\gamma_k)$ indeed coincide with the ones of $M(s,\pi)(f\cdot e^{\<s\gamma_k+\rho_{P_k},H_{P_k}(\cdot)\>})$ (\cite{moewal}, Lem.\ I.4.10 and Prop.\ II.1.7.(i)). (The behaviour of holomorphy of $M(s,\pi)$ as a function in $s$ is hence again independent of the choice of $w_k$, whence we have omitted it in its notation.)\\\\
It is clear that $M(s,\pi)$ breaks as a restricted tensor product $M(s,\pi)=\otimes'_\vv M(s,\pi_\vv)$ with respect to the normalized spherical vectors $f^\circ_\vv\in {\rm Ind}_{P_k(F_\vv)}^{G(F_\vv)}[\pi_\vv\otimes e^{\langle s\gamma_k,H_{P_k}(\cdot)\rangle}]$ (resp.\ $\tilde f^\circ_\vv\in {\rm Ind}_{P_k(F_\vv)}^{G(F_\vv)}[w_k(\pi)_\vv\otimes e^{\langle -s\gamma_k,H_{P_k}(\cdot)\rangle}]$) at those non-archimdean places $\vv\in S$, where $\pi$ (or, equivalently, $w_k(\pi)$) is unramified. 

\begin{prop}\label{prop:Msp}
Let $P_k=L_kN_k$ be a standard maximal parabolic $F$-subgroup of $G$, $1\leq k\leq r$, and let $\pi=\tau\hat\otimes\sigma$ be a unitary cuspidal automorphic representation of $L_k(\A_F)\cong \GL_k(\A_E)\times G_{n-2k}(\A_F)$. If $\tau_\infty$ is tempered and $\sigma$ satisfies the conditions of Thm.\ \ref{thm:qtcusp}, i.e., $\sigma$ is cohomological and $\Pi:=BC(\sigma)$ is the isobaric sum $\Pi=\Pi_1\boxplus...\boxplus\Pi_l$ of cuspidal automorphic representations $\Pi_i$, then the poles of $M(s,\pi)$ in the region $\Re e(s)\geq \tfrac12$ are precisely the ones of  
$$ L(s,\tau\times \Pi^\vee) \ L(2s, \tau, {\rm As}^{(-1)^n}).$$
\end{prop}
\begin{proof}
In order to describe the poles of $M(s,\pi)$, we follow the general approach to normalize all the local intertwining operators $M(s,\pi_\vv)$. At a place $\vv\in S$, define $\tau_\vv:=\otimes_{\ww|\vv}\tau_\ww$ and $\Pi_\vv:=\otimes_{\ww|\vv}\Pi_\ww$ and let
$$r(s,\pi_\vv):=\frac{ L(s,\tau_\vv\times \Pi^\vee_\vv) \ L(2s, \tau_\vv, {\rm As}^{(-1)^n})}{ L(s+1,\tau_\vv\times \Pi^\vee_\vv) \ L(2s+1, \tau_\vv, {\rm As}^{(-1)^n})}.$$
We observe that if $\vv$ splits in $E$, i.e., if there are two places $\ww_1$ and $\ww_2$ above $\vv$, then this explicitly yields the expression
\begin{equation}\label{eq:Ls}
r(s,\pi_\vv)=\frac{ L(s,\tau_{\ww_1}\times \Pi^\vee_{\ww_1}) \ L(s,\tau_{\ww_2}\times \Pi^\vee_{\ww_2}) \ L(2s, \tau_{\ww_1}\times\tau_{\ww_2},)}{  L(s+1,\tau_{\ww_1}\times \Pi^\vee_{\ww_1}) \ L(s+1,\tau_{\ww_2}\times \Pi^\vee_{\ww_2}) \ L(2s+1, \tau_{\ww_1}\times\tau_{\ww_2},)}
\end{equation}
of local Rankin-Selberg $L$-functions. Regardless of the nature of place $\vv\in S$, we define the local normalized intertwining operator at $\vv$ as
$$N(s,\pi_\vv):=r(s,\pi_\vv)^{-1}\cdot M(s,\pi_\vv).$$
Given our assumptions on $\sigma$, the local components $\Pi_\ww$ of $\Pi=BC(\sigma)$ are quasi-tempered at all non-achimedean places $\ww\in \Sigma$, see Thm.\ \ref{thm:qtcusp} above, and the same is true for $\tau_\ww$, as it is the local component of a unitary cuspidal automorphic representation (whence irreducible unitary and generic). Hence, if $\vv\in S$ is non-archimedean, then it was shown in Prop.\ A.7 of \cite{CKT} (see also their Rem.\ A.6) that $N(s,\pi_\vv)$ is holomorphic and not identically vanishing for $\Re e(s)\geq \tfrac12$. (In the proof of the aforementioned result, Ta\"ibi and Waldspurger refer to the local results of Mok, \cite{mok}, hence it is -- strictly speaking -- only applicable, if $G_{n-2k}$ is quasi-split over $F_\vv$. However, referring to \cite{zou} instead, their argument applies {\it verbatim} to all non-archimedean places.)\\\\
If $\vv\in S$ is archimedean, then Thm.\ \ref{thm:qtcusp} implies that $\pi_\vv\cong\tau_\vv\hat\otimes \sigma_\vv$ is a tempered representation of $L_k(F_\vv)$. Therefore, $M(s,\pi_\vv)$ is holomorphic and not identically zero for $\Re e(s)>0$, as it is the intertwining operator, whose image is the Langlands quotient attached to the Langlands datum $(P_k(F_\vv), \pi_\vv, s\cdot\gamma_k)$, cf.\  \cite{bowa}, IV.4.2--IV.4.4. The normalizing factor $r(s,\pi_\vv)$ -- essentially a finite product of $\Gamma$-functions -- is then easily seen to be holomorphic and non-zero for $\Re e(s)>0$: Indeed, let us first assume that $\vv\in S^{\sf in}_\R$, which, by \eqref{eq:Ls}, is the only instance, where the archimedean Asai-$L$-functions does not collapse to the Rankin-Selberg $L$-function. As usual, we abbreviate $\Gamma_{\mathbb R}(s):=\pi^{-\tfrac{s}{2}}\Gamma(\tfrac{s}{2})$, $\Gamma_{\mathbb C}(s):=2(2\pi)^{-s}\Gamma(s)$. Writing $\tau_\ww= z^{a_1}\overline{z}^{-a_1}\times ...\times z^{a_k}\overline{z}^{-a_k}$ and $\Pi_\ww=z^{b_1}\overline{z}^{-b_1}\times ...\times z^{b_{n-2k}}\overline{z}^{-b_{n-2k}}$ with $a_i,b_i\in \R$, cf.\ \S \ref{sect:tempGLcoh}, we get 
\begin{equation}\label{eq:Larch}
L(s,\tau_\ww\times\Pi_\ww^\vee)
=
\prod_{i=1}^k\prod_{j=1}^{n-2k}
\Gamma_{\mathbb C}
\left(
s+{|a_i-b_j|}
\right).
\end{equation}
and
$$
L(s,\tau_\ww,{\rm As}^{(-1)^n})
=
\prod_{1\le i<j\le k}
\Gamma_{\mathbb C}
\left(
s+{|a_i-a_j|}
\right)
\prod_{i=1}^k
\Gamma_{\mathbb R}(s+n-2\lfloor\tfrac{n}{2}\rfloor),
$$
which immediately implies that $r(s,\pi_\vv)$ is holomorphic and non-vanishing for $\Re e(s)>0$, if $\vv\in S^{\sf in}_\R$. Using \eqref{eq:Ls} and \eqref{eq:Larch} one reduces the remaining cases of split real and complex places $\vv\in S_\infty$ to the well-known properties of local archimedean Rankin-Selberg $L$-functions, which we leave to the reader. In summary, we have shown that the normalized intertwining operator $N(s,\pi_\vv)$ is holomorphic and not identically zero for $\Re e(s)>0$ at all archimedean places $\vv\in S$.\\\\
Forming the product over all $\vv\in S$ and observing that at a place, where $\pi_\vv$ is unramified, the normalized intertwining operator $N(s,\pi_\vv)$ just maps the normalized spherical vector $f^\circ_\vv\in {\rm Ind}_{P_k(F_\vv)}^{G(F_\vv)}[\pi_\vv\otimes e^{\langle s\gamma_k,H_{P_k}(\cdot)\rangle}]$ to its counterpart $\tilde f^\circ_\vv\in {\rm Ind}_{P_k(F_\vv)}^{G(F_\vv)}[w_k(\pi)_\vv\otimes e^{\langle -s\gamma_k,H_{P_k}(\cdot)\rangle}]$, we obtain that for $\Re e(s)\geq\frac12$ the poles of $M(s,\pi)$ are the same as of
$$r(s,\pi):=\prod_{\vv\in S}r(s,\pi_\vv)=\frac{ L(s,\tau\times \Pi^\vee) \ L(2s, \tau, {\rm As}^{(-1)^n})}{ L(s+1,\tau\times \Pi^\vee) \ L(2s+1, \tau, {\rm As}^{(-1)^n})}.$$
Since the denominator 
$$L(s+1,\tau\times \Pi^\vee) \ L(2s+1, \tau, {\rm As}^{(-1)^n})=\prod_{i=1}^l L(s+1,\tau\times \Pi_i^\vee) \cdot L(2s+1, \tau, {\rm As}^{(-1)^n})$$ 
is holomorphic and non-zero in the region considered, $\Re e(s)\geq\frac12$, see \cite{cogdel}, Thm.\ 4.2 and Thm.\ 4.3 (for the Rankin-Selberg $L$-functions) and \cite{grbacshahidi}, Thm.\ 4.3 (for the Asai $L$-functions), the result finally follows.
\end{proof}

\begin{rem}
We point out that, if $n-2k=0$, then the above result is simply to be read that the first factor in the product $ L(s,\tau\times \Pi^\vee) \ L(2s, \tau, {\rm As}^{(-1)^n})$ collapses into $1$. See also \cite{grbacshahidi}, Thm.\ 2.1, where the case of $k=\lfloor \frac{n}{2}\rfloor$ was considered in details. This is due to our convenient (but maybe slightly less conventional) choice of a basis vector of $ \check\a_{P,\C}$, which we have used in order to obtain the complex parameter $s$ in the intertwining operator $M(s,\pi)$: For this we have chosen $\gamma_k$ of instead of $\tilde\alpha_k:=(\rho_{P_k},\alpha^\vee_k)^{-1}\cdot\rho_{P_k}$, as it would be suggested by \cite{shahidi2}, p.\ 552. Our choice amounts to a ``uniform shape'' of the normalizing factors $r(s,\pi)$, i.e., independent of $k$, and enabled us to use \cite{CKT}, where the same approach was taken (in fact, as also in the aforementioned \cite{grbacshahidi}). We remark that passing from our choice of basis vector $\gamma_k$ to $\tilde\alpha_k$ would amount to a rescaling the argument $s$ of $M(s,\pi)$ by the (unusual) factor $\tfrac14$, if $k=r$ and $\ell_{P_0}\geq 1$, and by $\tfrac12$, else, since
\begin{equation*}\label{eq:alphak}
\tilde\alpha_k=\left\{\begin{array}{ll}
 \tfrac12 (\varepsilon_1+\cdots+\varepsilon_k) & \textrm{if $k=r$ and $\ell_{P_0}\geq 1$} \\
 \varepsilon_1+\cdots+\varepsilon_k& \textrm{else.}
\end{array}
\right.
\end{equation*}
\end{rem}
\noindent 
Next, for $n,k\geq 1$ as above, recall the quadratic character $\epsilon: \GL_k(F)\backslash \GL_k(\A_F)\ra\C^*$, from \S \ref{sect:nf}. As usual, a cuspidal automorphic representation $\tau$ of $\GL_k(\A_E)$ is called $\epsilon$-{\it distinguished}, if 
$$\Hom_{\GL_k(\A_F)}(\tau|_{\GL_k(\A_F)},\epsilon)\neq\{0\}.$$ 
With this notion at hand, we get

\begin{cor}\label{cor:holMspi}
Under the assumptions of Prop.\ \ref{prop:Msp}, the poles of $M(s,\pi)$ for a real $s{\geq\frac12}$ are at $s=\tfrac12,1$. The pole at $s=\tfrac12$ occurs, if and only if $\tau$ is conjugate self-dual, $\epsilon$-distinguished and $L(\tfrac12,\tau\times \Pi^\vee)\neq 0$; whereas the pole at $s=1$, occurs, if and only if $\tau=\Pi_i$ for some $1\leq i\leq l$. Both poles are, if they occur, simple.
\end{cor}
\begin{proof}
Let $s\in\R_{\geq 0}$, and let us factorize as above 
\begin{equation}\label{eq:Lss}
L(s,\tau\times \Pi^\vee) \ L(2s, \tau, {\rm As}^{(-1)^n})=\prod_{i=1}^l L(s,\tau\times \Pi_i^\vee) \cdot L(2s, \tau, {\rm As}^{(-1)^n}).
\end{equation}
We first look at the factors on the right hand side of \eqref{eq:Lss}: It is well-known that $L(s,\tau\times\Pi_i^\vee)$ is entire, if $k\neq n_i$, and otherwise has simple poles precisely at $s=0,1$, which occur, if and only if $\tau=\Pi_i$, see \cite{cogdel}, Thm.\ 4.2. By Thm.\ 4.3, {\it ibidem} it is non-vanishing for $s\geq 1$. Similarly, $L(s, \tau, {\rm As}^{(-1)^n})$ is entire, if $\tau$ is not conjugate self-dual and otherwise has a simple poles precisely at $s=0,1$, which occur, if and only if $\tau$ is $\epsilon$-distinguished, see \cite{flickerzinoviev} in combination with \cite{grbacshahidi}, Thm.\ 4.3. Indeed, for this characterization of the existence of poles through $\epsilon$-distinction, we observe that by the result of Flicker-Zinoviev, in the form recalled for instance in
\cite{MOY}, Thm.\ 1.1.\ -- where, in contrast to the main result of \cite{flickerzinoviev}, the case of {\it completed} $L$-functions is treated -- a cuspidal
automorphic representation $\theta$ of \(\mathrm{GL}_k(\mathbb A_E)\) is \(\triv_{\mathrm{GL}_k(\mathbb A_F)}\)-distinguished, if and only if \(L(s,\theta, {\rm As}^+)\) has a pole at \(s=1\). Recalling our choice of a Hecke character \(\eta_{E/F} :E^*\backslash \A^*_E\ra\C^*\), which extends $\varepsilon_{E/F}$, i.e., such that $\left.\eta_{E/F}\right|_{\mathbb A_F^\times}=\varepsilon_{E/F}$, from Sect.\ \ref{sect:nf}, the \(\varepsilon_{E/F}\)-distinction of \(\tau\) is equivalent to the \(\triv_{\mathrm{GL}_k(\mathbb A_F)}\)-distinction of \(\tau\otimes\eta_{E/F}^{-1}\), while $L(s,\tau\otimes\eta_{E/F}^{-1},{\rm As}^+)=L(s,\tau,{\rm As}^-)$, which shows the just made claim about $\epsilon$-distinction. We recall furthermore that by \cite{grbacshahidi}, Thm.\ 4.3, $L(s, \tau, {\rm As}^{(-1)^n})$ is also known to be non-vanishing for $s\geq 1$.\\\\ 
Now, let $s\geq \tfrac12$. Looking at the left hand side of \eqref{eq:Lss}, by what we have just observed, the factor $L(s,\tau\times \Pi^\vee)$ has a pole precisely at $s=1$ and it occurs, if and only if $\tau=\Pi_i$ for some $1\leq i\leq l$. Furthermore, by what we have just recalled, this singularity will survive in the product \eqref{eq:Lss}, i.e., the pole of $L(s,\tau\times \Pi^\vee)$ at $s=1$ will not be cancelled by some zero of the second factor in \eqref{eq:Lss}. Moreover, this pole is simple, which is a consequence of our Thm.\ \ref{thm:qtcusp}, which implies that $\Pi$ being the cohomological base change of $\sigma$, it has no repeating isobaric summands, i.e., $\Pi_i\neq\Pi_j$, for $i\neq j$.\\\\ In addition, our above observations show that $L(2s, \tau, {\rm As}^{(-1)^n})$ has a pole precisely at $s=\tfrac12$, and this singularity occurs, if and only if $\tau={}^c\tau^\vee$ is $\epsilon$-distinguished. Clearly, it remains a pole of $L(s,\tau\times \Pi^\vee) \ L(2s, \tau, {\rm As}^{(-1)^n})$, if and only if $L(\tfrac12,\tau\times \Pi^\vee)\neq 0$, in which case it is simple by the just quoted \cite{grbacshahidi}, Thm.\ 4.3.\\\\ Applying Prop.\ \ref{prop:Msp} now shows the claim. 
\end{proof}

\begin{rem}
We refer to \cite{ckpssh, mok, zou} for the relation of $\epsilon$-distinction and base change. 
\end{rem}

\begin{thm}\label{thm:Poles}
For a standard maximal parabolic $F$-subgroup $P_k$, $1\leq k\leq r$, let $\varphi_{P_k}$ be an associate class, represented by a cuspidal automorphic representation $\pi_{\lambda_0}=\pi\otimes e^{\langle \lambda_0,H_{P}(\cdot)\rangle}$ of $L_k(\A_F)$ as in Sect.\ \ref{sect:pc}. We write $\pi=\tau\hat\otimes\sigma$ for a unitary cuspidal automorphic representation $\tau$ of $\GL_k(\A_E)$ and a unitary cohomological (cf.\ Lem.\ \ref{lem:varphis}) cuspidal automorphic representation $\sigma$ of $G_{n-2k}(\A_F)$ and we assume that $\Pi=BC(\sigma)$ satisfies the equivalent conditions of Thm.\ \ref{thm:qtcusp}, i.e., $\Pi=\Pi_1\boxplus...\boxplus\Pi_l$ for cuspidal automorphic representations $\Pi_i$, $1\leq i\leq l$. 

Then, there exists a smooth $K_\infty$-finite section $f\in I^G_P(\pi)$, such that the Eisenstein series $E_{P}(f,\lambda)$ has a pole at the attached point $\lambda_0\in\check\a_{P,\C}$, if and only if either
\begin{enumerate}
\item \underline{$\lambda_0=\tfrac12\gamma_k$}: $\tau$ is conjugate self-dual and $\epsilon$-distinguished and $L(\tfrac12,\tau\times \Pi^\vee)\neq 0$; or
\item \underline{$\lambda_0=\gamma_k$}: $\tau=\Pi_i$ for some $1\leq i\leq l$.
\end{enumerate}
In both cases, the poles of $E_P(f,\lambda)$ at $\lambda=\lambda_0$ are simple.
\end{thm}
\begin{proof}
Let us write $\lambda_0=s_0\cdot\gamma_k$. Since the ideal $\J$ of $\mathcal Z(\g)$ entering the definition of $\varphi_{P_k}$, cf.\ Sect.\ \ref{sect:pc}, was chosen to be the one, which annihilates the contragredient representation $\EE_\mu^\vee$ of $\EE_\mu$, we obtain that $s_0$ is a non-negative half-integer, i.e., $s_0\in \tfrac12 \Z_{\geq 0}$: In order to see this, we firstly  observe that, as explained on p.\ 772 and 774 of \cite{schwfr} or p.\ 1077 of \cite{grobner-EisRes}, $\lambda_0\in\overline{\check\a_{P_k}^{+}}$, so $s_0$ must be real and non-negative, i.e., $s_0\in\R_{\geq 0}$. Secondly, recalling the two facts that (i) $\pi\otimes e^{\langle \lambda_0+\rho_P,H_{P}(\cdot)\rangle}$ is cohomological with respect to an irreducible direct summand of the semisimple $L_{k,\infty}$-module $H^*(\n_{k,\infty},\EE_\mu)$, which was shown in the course of the proof of Lem.\ \ref{lem:varphis}, and that (ii) $\pi_\infty$ is irreducible and unitary (and hence conjugate self-dual), we obtain from \cite{bowa} I, Cor.\ 4.2 and \cite{bocas}, Lem.\ 1.3 that $\lambda_0\in \tfrac12 X^*(L_k)$. (Invoking condition (3) from Sect.\ \ref{sect:pc} above, the detailed argument for the latter claim is also carried out in the proof of Thm.\ 2.1 of \cite{CKT}.) Finally, our \eqref{halfint} implies $s_0\in \tfrac12 \Z_{\geq 0}$ as desired.\\\\ 
Keeping this in mind, we invoke \cite{moewal}, Prop.\ IV.1.11, which shows that $M(s,\pi)$ and (hence) $E_{P_k}(f,s\gamma_k)$ is holomorphic at $s=0$. We are hence reduced to studying the behaviour of holomorphy of the Eisenstein series $E_{P_k}(f,s\gamma_k)$ at points $s=s_0\geq\tfrac12$. However, we already observed that the poles of the Eisenstein series $E_{P_k}(f,s\gamma_k)$ indeed coincide with the ones of $M(s,\pi)(f\cdot e^{\<s\gamma_k+\rho_{P_k},H_{P_k}(\cdot)\>})$. Therefore, the result follows from Cor.\ \ref{cor:holMspi} (and Lem.\ \ref{lem:varphis}, which ensures that $\tau_\infty$ is tempered and $\sigma_\infty$ is cohomological, and hence Cor.\ \ref{cor:holMspi} can be applied to $\pi$) and the fact, already observed in course of the proof of Prop.\ \ref{prop:Msp}, that $N(s,\pi)$ and hence $M(s,\pi)$, does not vanish identically.
\end{proof}

\section{Kostant data}

\subsection{Highest weight modules and evaluation points}\label{sect:weylgrps}

For the local set of absolute roots $\Delta(\g_{\vv,\C},\h_{\vv,\C})$ at $\vv\in S_\infty$, cf\ Sect.\ \ref{sect:rlgrps}, we write $W_\vv:=W(\g_{\vv,\C},\h_{\vv,\C})$ for the corresponding Weyl group generated by the attached reflections. Given a standard parabolic $F$-subgroup $P$ of $G$, the set of {\it Kostant representatives} $W_\vv^{P}$ at $\vv\in S_\infty$ is then defined as the set of all $w_\vv\in W_\vv$ such that $w_\vv^{-1}(\alpha)>0$ for all simple roots $\alpha_\vv\in\Delta(\l_{\vv,\C},\h_{\vv,\C})$, cf.\ \cite{bowa}, III.1.4. For $\vv\in S_\C$, $\g_{\vv,\C}\cong \g\l_n(\C)\oplus \g\l_n(\C)$, whence $W_\vv^{P}$ splits as a product of two sets $W_\vv^{P}=W_{\iota_\vv}^{P}\times W_{\bar\iota_\vv}^{P}$ of formally the same shape (namely of the one at a $\vv\in S_\R$, if there is any such place). We get

\begin{lem}\label{lem:Kostant}
Let $P_k$ be the standard maximal parabolic $F$-subgroup of $G$ corresponding to $\alpha_k\in\Delta^0_F$, $1\leq k\leq r$. Let $\vv\in S_\R$. Then, $W_\vv=W(\g_{\vv,\C},\h_{\vv,\C})\cong \mathfrak S_n$ may be identified with the symmetric group of permutations of the set $\{1,...,n\}$ and 
\begin{equation}\label{eq:Kostant}
W_\vv^{P_k}=
\left\{
w \in \mathfrak S_n \;\middle|\;
\begin{aligned}
& w^{-1}(1) < \cdots < w^{-1}(k), \\
& w^{-1}(k+1) < \cdots < w^{-1}(n-k), \\
& w^{-1}(n-k+1) < \cdots < w^{-1}(n)
\end{aligned}
\right\}.
\end{equation}
Hence, $W_\vv^{P_k}$ is parametrized as a subset of $\mathfrak S_n$ by pairs $(I,J)$ of disjoint subsets
$I:=\{i_1, \cdots, i_k\}\subseteq\{1,\dots,n\}$ and $J:=\{j_1,...,j_{k}\}\subseteq\{1,\dots,n\}$, with $i_1<...<i_k$ and $j_1<...<j_k$, and complement
$\{1,\dots,n\}\setminus (I\cup J)=:\{r_1, ..., r_{n-2k}\}$, with $r_1< ...< r_{n-2k}$ via
$$
w_{I,J}(i_\ell)=\ell,\qquad w_{I,J}(j_\ell)=n-k+\ell, \qquad w_{I,J}(r_\ell)=k+\ell.
$$
If $\vv\in S_\C$, then $W_\vv^{P_k}$ is given by the product of two identical, but independent copies of sets of type \eqref{eq:Kostant}.
\end{lem}
\begin{proof}
This follows as in \cite{gotsgro}, Prop.\ 2.1. 
\end{proof}
\noindent We let $W^{P}:=\prod_{\vv\in S_\infty} W^{P_\vv}$ and define an affine action of $W^P$ on $\check\h_{\infty,\C}$ by 
$$w\cdot\mu:=\mu_w:=w(\mu+\rho)-\rho.$$ 
Clearly, $\mu_w=(\mu_{w_\vv})_{\vv\in S_\infty},$ where $\mu_{w_\vv}:=w_\vv\cdot\mu_\vv:=w_\vv(\mu_\vv+\rho_\vv)-\rho_\vv$. Given a dominant integral weight $\mu=(\mu_\vv)_{\vv\in S_\infty}\in \check\h_{\infty}$ and $w\in W^P$, we let $\EE_{\mu_w}=\otimes_{\vv\in S_\infty} \EE_{\mu_{w_\vv}}$ be the irreducible representation of $L_{P,\infty}$ of highest weight $\mu_w$. These are precisely the irreducible direct summands of the semisimple $L_{P,\infty}$-module $H^*(\n_{P,\infty},\EE_\mu)$, mentioned already in the proofs of Lem.\ \ref{lem:varphis} and Thm.\ \ref{thm:Poles} above, which is a celebrated result of Kostant, cf.\ \cite{kostant}, Thm.\ 5.14.
In order to qualify to be a highest weight module, with respect to which a twisted representative $(\pi_{\lambda_0}\otimes e^{\langle \rho_P,H_{P}(\cdot)\rangle})_\infty$ of an associate class $\varphi_P$ of cuspidal automorphic representations of $L_P(\A)$ has non-trivial cohomology, \cite{bowa} I, Cor.\ 4.2 and \cite{bocas}, Lem.\ 1.3, enforce that we must have
$$\EE_{\mu_w}|_{\m_{P,\infty}}\cong \bar \EE^{\vee}_{\mu_w}|_{\m_{P,\infty}},$$
where $\bar \EE^{\vee}_{\mu_w}|_{\m_{P,\infty}}$ denotes the complex conjugate, contragredient representation of the $\m_{P,\infty}$-module $\EE_{\mu_w}|_{\m_{P,\infty}}$. (Recall that $\m_{P,\infty}$ denotes the real Lie algebra of $M_{P,\infty}=\bigcap_\xi \ker |\xi|$, $\xi:L_{P,\infty}\rightarrow\C^*$ running through the continuous characters.) In particular, we obtain
\begin{equation}\label{eq:complselfcont}
\EE_{\mu_{w_\vv}}|_{\m_{P,\vv}}\cong \bar \EE^{\vee}_{\mu_{w_\vv}}|_{\m_{P,\vv}}
\end{equation}
for all archimedean places $\vv\in S_\infty$.\\\\
Analogously, for $w\in W^P$ we let
$$\lambda_w:=-w(\mu+\rho)|_{\a_{P_\infty}},$$
which factors als $\lambda_w=(\lambda_{w_\vv})_{\vv\in S_\infty},$ where $\lambda_{w_\vv}:=-w_\vv(\mu_\vv+\rho_\vv)|_{\a_{P_\vv}}$. In order to serve as (the derivative) of the second entry of an associate class of cuspidal automorphic representations of $L_P(\A_F)$, we must have 
\begin{equation}\label{eq:proj=}
{\rm pr}_{\check\a_{P_\vv}\ra \check\a_{P}}(\lambda_{w_\vv})={\rm pr}_{\check\a_{P_{\vv'}}\ra \check\a_{P}}(\lambda_{w_{\vv'}})
\end{equation}
for all archimedean places $\vv,\vv'\in S_\infty$, i.e., the union
$$\bigcup_{\vv\in S_\R}\{\lambda_{w_{\iota_\vv}}\}\cup \bigcup_{\vv\in S_\C}\{\tfrac12 (\lambda_{w_{\iota_\vv}}+\lambda_{w_{\bar\iota_\vv}})\}$$
is a singleton, or, otherwise put, contains only one single vector, which is in fact an element of $\overline{\check\a_P^{+}}$. See also \cite{grobner-EisRes}, Prop.\ 10. By abuse of notation, we will again denote this vector by $\lambda_w$.\\\\
Let now $P=P_k$ be maximal parabolic, $1\leq k\leq r$, and write $\lambda_w= s_w\cdot \gamma_k$ for $w\in W^{P_k}$. Condition \eqref{eq:proj=} now translates into the following equality of sets,
$$\bigcup_{\vv\in S_\R}\{s_{w_{\iota_\vv}}\}\cup \bigcup_{\vv\in S_\C}\{\tfrac12 (s_{w_{\iota_\vv}}+s_{w_{\bar\iota_\vv}})\}=\{s_w\},$$
where $s_{w_{\iota_\vv}}$ is defined as the respective coefficient in  
$$-w_{\iota_\vv}(\mu_{\iota_\vv}+\rho_{\iota_\vv})|_{\a_{P_{\iota_\vv}}}=s_{w_{\iota_\vv}}\cdot \gamma_k.$$
Using our explicit Lem.\ \ref{lem:Kostant}, resp.\ \eqref{eq:Kostant}, for every ${\iota_\vv}: F\hra\C$, we obtain $W_{\iota_\vv}\cong S_n$ and a \(w_{\iota_\vv}\in W_{\iota_\vv}^{P_k}\) is given by a partition
$$
\{1,\dots,n\}=I\sqcup R\sqcup J
$$
with
$$
|I|=|J|=k,\qquad |R|=n-2k.
$$
Writing explicitly,
$$
\mu_{\iota_\vv}=(\mu_{{\iota_\vv},1},\dots,\mu_{{\iota_\vv},n}) \quad {\rm and}\quad \rho_{\iota_\vv}=\left(\tfrac{n-1}{2},\tfrac{n-3}{2},\dots,\tfrac{1-n}{2}\right)
$$
for a dominant integral highest weight of an irreducible finite-dimensional complex representation of $\g_{{\iota_\vv},\C}\cong \g\l_n(\C)$ and the half sum of positive absolute local roots at ${\iota_\vv}$, respectively, we get
$$
s_{w_{\iota_\vv}}
=
\frac{1}{2k}
\left(\sum_{j\in J}(\mu_{{\iota_\vv},j}-j)
-
\sum_{i\in I}(\mu_{{\iota_\vv},i}-i)\right),
$$
i.e., if we finally let $w\in W^{P_k}$ be represented in suggestive notation by an $S_\infty$-tuple of such partitions, $((I_\vv,J_\vv)_{\vv\in S_\R},((I_{\iota_\vv},J_{\iota_\vv}),(I_{\bar\iota_\vv},J_{\bar\iota_\vv}))_{\vv\in S_\C})$, then \fontsize{7.5}{10}\selectfont
\begin{equation}\label{sw}
\boxed{s_{w}
=
\begin{cases}
\displaystyle
\frac{1}{2k}
\left(\sum_{j_\vv\in J_\vv}(\mu_{{\vv},j_\vv}-j_\vv)
-
\sum_{i_\vv\in I_\vv}(\mu_{{\vv},i_\vv}-i_\vv)\right),
&\forall \vv\in S_\R,\\[16pt]
\displaystyle
\frac{1}{4k}
\left(\sum_{j_{\iota_\vv}\in J_{\iota_\vv}}(\mu_{{\iota_\vv},j_{\iota_\vv}}-j_{\iota_\vv})+\sum_{j_{\bar\iota_\vv}\in J_{\bar\iota_\vv}}(\mu_{{\bar\iota_\vv},j_{\bar\iota_\vv}}-j_{\bar\iota_\vv})
-
\left(\sum_{i_{\iota_\vv}\in I_{\iota_\vv}}(\mu_{{\iota_\vv},i_{\iota_\vv}}-i_{\iota_\vv})
+
\sum_{i_{\bar\iota_\vv}\in I_{\bar\iota_\vv}}(\mu_{{\bar\iota_\vv},i_{\bar\iota_\vv}}-i_{\bar\iota_\vv})\right)\right),
& \forall\vv\in S_\C.
\end{cases}}
\end{equation}\normalsize
If $\mu$ is the highest weight of an algebraic representation of $G_\infty$, it follows immediately from this description, that $s_w\in\Q$ for all $w\in W^{P_k}$, $1\leq k\leq r$. Moreover, if one also assumes that $\mu_w$ is the highest weight of a representation of $L_{P_k,\infty}$, with respect to which a twisted representative $(\pi_{\lambda_0}\otimes e^{\langle \rho_{P_k},H_{P_k}(\cdot)\rangle})_\infty$ of an associate class $\varphi_{P_k}$ of cuspidal automorphic representations of $L_{P_k}(\A_F)$ has non-trivial cohomology, then the conjugate self-duality condition \eqref{eq:complselfcont} implies the aforementioned much stronger property that 
$$s_w\in\tfrac12\Z.$$
In view of Thm.\ \ref{thm:Poles}, we shall be particularly interested in Kostant representatives $w\in W^{P_k}$, $1\leq k\leq r$, which satisfy \eqref{eq:complselfcont} and for which $s_w\in\{\tfrac12,1\}$

\subsection{A non-trivial example - Part I}
Before we continue with the construction of non-trivial residual Eisenstein cohomology classes, let us consider a non-trivial example, in order to illustrate the phenomena, which one may encounter when computing the points of evaluation $s_w\in\Q$, $w\in W^{P_k}$ as above. As just indicated, we shall particularly be interested in determining the Kostant representatives, $w\in W^{P_k}$, $1\leq k\leq r$, which satisfy \eqref{eq:complselfcont} and which give rise to $s_w\in\{\tfrac12,1\}$.

\subsubsection{}\label{sect:extex}
Let $F=\Q(\sqrt[3]{2})$ and let $E=F(\sqrt{-1})$. Then \(F\) has one real place, denoted $\vv_\R$, which is inert, and one complex place, denoted $\vv_\C$. Let \(V_6:=E^6\) with a hyperbolic hermitian form $h$, e.g., $h(x,y)=x^t\cdot J\cdot c(y)$, with $J$ being the matrix having $1$'s on the skew diagonal and $0$'s elswehere, and consider in this subsection only the attached unitary group $G=U(V_6,h)$. Then, $G$ is quasisplit over $F$ of $rk_F(G)=r=3$ and
$$
G(F_{\vv_{\R}})\cong U(3,3),
\qquad
G(F_{\vv_{\C}})\cong \GL_6(\C)
$$
as real Lie groups. Accordingly, we get $3$ maximal standard parabolic $F$-subgroups \(P_k\), \(k=1,2,3\) and we have, $\rho
=
\left(\frac52,\frac32,\frac12,-\frac12,-\frac32,-\frac52\right).$
Throughout this example, we use the trivial coefficient system, i.e., $\mu=0$. At the real archimedean place, a Kostant representative $w_{\vv_\R}\in W^{P_k}_{\vv_\R}$ is given by a partition
$$
\{1,2,3,4,5,6\}=I\sqcup R\sqcup J
$$
with
$$
|I|=|J|=k,\qquad |R|=6-2k.
$$
Define
$$
C(w_{\vv_\R}):=\sum_{i\in I} i-\sum_{j\in J} j.
$$
With respect to the basis \(\gamma_k\), the corresponding point of evaluation is $s_{w_{\vv_\R}}=\frac{1}{2k} C(w_{\vv_\R}).$ Therefore,
$$
s_{w_{\vv_\R}}=\tfrac12 \Longleftrightarrow C(w_{\vv_\R})=k,
\qquad
s_{w_{\vv_\R}}=1 \Longleftrightarrow C(w_{\vv_\R})=2k.
$$
At the complex archimedean place $\vv_\C$, a Kostant representative $w_{\vv_\C}\in W^{P_k}_{\vv_\C}$ is a pair
$$
w_{\vv_\C}=(w_+,w_-),
$$
with $w_\pm$ being represented by $(I_\pm,R_\pm,J_\pm)$. Putting
$$
C(w_\pm):=\sum_{i\in I_\pm}i-\sum_{j\in J_\pm}j, 
$$
we got to consider the point of evaluation $s_{w_{\vv_\C}}=\frac{1}{4k} (C(w_+)+C(w_-))$. Thus,
$$
s_{w_{\vv_\C}}=\tfrac12 \Longleftrightarrow C(w_+)+C(w_-)=2k,
\qquad
s_{w_{\vv_\C}}=1 \Longleftrightarrow C(w_+)+C(w_-)=4k.
$$

\subsubsection{Tables of representatives}\label{sect:tablesw}

The following three tables show the complete lists of all Kostant representatives $w\in W^{P_k}$, $k=1,2,3$, respectively, giving rise to evaluation points $s_w\in\{\tfrac12,1\}$ and for which the conjugate self-duality condition \eqref{eq:complselfcont} holds. Recall that in our concrete example, a Kostant representative splits as a tuple $w=(w_{\vv_{\mathbb R}},w_{\vv_{\mathbb C}})$ of length $\ell(w)=\ell(w_{\vv_{\mathbb R}})+\ell(w_{\vv_{\mathbb C}})$.\\\\
\small
\underline{$k=1$:} $L_1\cong {\rm Res}_{E/F}(\GL_1) \times U(V_4,h)$. 
\begin{center}
\tiny
\setlength{\tabcolsep}{2pt}\renewcommand{\arraystretch}{1.06}
\begin{longtable}{r p{1.78cm} c p{4cm} c | p{1.78cm} c p{4.15cm} c}
\toprule
& \multicolumn{4}{c|}{$s_{w_\bullet}=\frac12$} & \multicolumn{4}{c}{$s_{w_\bullet}=1$}\\
\midrule
& \(w_{\vv_{\mathbb R}}\)  & $\ell(w_{\vv_{\mathbb R}} )$ & \quad $w_{\vv_{\mathbb C}}=(w_+, w_-)$ & $\ell(w_{\vv_{\mathbb C}})$ & \ \(w_{\vv_{\mathbb R}}\)  & $\ell(w_{\vv_{\mathbb R}} )$ & \quad $w_{\vv_{\mathbb C}}=(w_+, w_-)$ & \(\ell(w_{v_{\mathbb C}})\)\\
\midrule
\endhead
& $(2,3,4,5,6,1)$ & 5 & $(2,3,4,5,6,1)$, $(2,3,4,5,6,1)$ & 10 &  \ \  $\emptyset$ &  \  $\emptyset$& $(2,3,4,6,5,1)$; $(6,2,1,3,4,5)$ & 12\\
& $(2,3,6,1,4,5)$ & 5 & $(2,3,4,5,6,1)$, $(6,1,2,3,4,5)$ & 10 &  &  & $(2,3,6,4,1,5)$; $(2,6,3,1,4,5)$ & 12\\
 & $(6,1,2,3,4,5)$ & 5 & $(2,3,4,6,1,5)$, $(2,6,1,3,4,5)$ & 10 &  &  & $(2,6,3,1,4,5)$; $(2,3,6,4,1,5)$ & 12\\
 &  &  & $(2,3,6,1,4,5)$, $(2,3,6,1,4,5)$ & 10 &  &  & $(6,2,1,3,4,5)$; $(2,3,4,6,5,1)$ & 12\\
 &  &  & $(2,6,1,3,4,5)$, $(2,3,4,6,1,5)$ & 10 &  &  & $(1,6,2,3,4,5)$; $(6,2,3,4,5,1)$ & 13\\
 &  &  & $(6,1,2,3,4,5)$, $(2,3,4,5,6,1)$ & 10 &  &  & $(2,3,4,5,1,6)$; $(6,2,3,4,5,1)$ & 13\\
 &  &  & $(6,1,2,3,4,5)$, $(6,1,2,3,4,5)$ & 10 &  &  & $(6,2,3,4,5,1)$; $(1,6,2,3,4,5)$ & 13\\
 &  &  & $(1,2,6,3,4,5)$, $(6,2,3,4,1,5)$ & 11 &  &  & $(6,2,3,4,5,1)$; $(2,3,4,5,1,6)$ & 13\\
 &  &  & $(2,3,4,1,5,6)$, $(2,6,3,4,5,1)$ & 11 &  &  &  & \\
 &  &  & $(2,6,3,4,5,1)$, $(2,3,4,1,5,6)$ & 11 &  &  &  & \\
 &  &  & $(6,2,3,4,1,5)$, $(1,2,6,3,4,5)$ & 11 &  &  &  & \\
\bottomrule
\end{longtable}
\normalsize
\end{center}
\newpage\small\enlargethispage{2cm}
\noindent\underline{$k=2$:} $L_2\cong {\rm Res}_{E/F}(\GL_2) \times U(V_2,h)$. 
\begin{center}
\tiny
\setlength{\tabcolsep}{2pt}\renewcommand{\arraystretch}{1.06}
\begin{longtable}{r p{1.78cm} c p{4cm} c | p{1.78cm} c p{4.15cm} c}
\toprule
& \multicolumn{4}{c|}{$s_{w_\bullet}=\frac12$} & \multicolumn{4}{c}{$s_{w_\bullet}=1$}\\
\midrule
& \(w_{\vv_{\mathbb R}}\)  & $\ell(w_{\vv_{\mathbb R}} )$ & \quad $w_{\vv_{\mathbb C}}=(w_+, w_-)$ & $\ell(w_{\vv_{\mathbb C}})$ & \ \(w_{\vv_{\mathbb R}}\)  & $\ell(w_{\vv_{\mathbb R}} )$ & \quad $w_{\vv_{\mathbb C}}=(w_+, w_-)$ & \(\ell(w_{v_{\mathbb C}})\)\\
\midrule
\endhead
1 & $(3,4,5,1,6,2)$ & 7 & $(3,4,5,1,2,6)$ ; $(3,4,5,6,1,2)$ & 14 & $(3,4,5,6,1,2)$ & 8 & $(3,4,5,6,1,2)$ ; $(3,4,5,6,1,2)$ & 16 \\
2 & $(3,5,1,4,6,2)$ & 7 & $(3,4,5,1,2,6)$ ; $(5,6,1,2,3,4)$ & 14 & $(3,5,6,1,2,4)$ & 8 & $(3,4,5,6,1,2)$ ; $(5,6,1,2,3,4)$ & 16 \\
3 & $(3,5,1,6,2,4)$ & 7 & $(3,4,5,1,6,2)$ ; $(3,4,5,1,6,2)$ & 14 & $(5,6,1,2,3,4)$ & 8 & $(3,5,6,1,2,4)$ ; $(3,5,6,1,2,4)$ & 16 \\
4 & $(5,1,3,4,6,2)$ & 7 & $(3,4,5,1,6,2)$ ; $(5,1,6,2,3,4)$ & 14 & $(5,1,6,3,4,2)$ & 9 & $(5,6,1,2,3,4)$ ; $(3,4,5,6,1,2)$ & 16 \\
5 & $(5,1,3,6,2,4)$ & 7 & $(3,4,5,6,1,2)$ ; $(3,4,5,1,2,6)$ & 14 & $(5,3,1,6,4,2)$ & 9 & $(5,6,1,2,3,4)$ ; $(5,6,1,2,3,4)$ & 16 \\
6 & $(5,1,6,2,3,4)$ & 7 & $(3,4,5,6,1,2)$ ; $(5,1,2,6,3,4)$ & 14 & $(5,3,4,1,6,2)$ & 9 & $(3,5,4,1,6,2)$ ; $(5,6,1,3,2,4)$ & 17 \\
7 & $(1,5,6,3,4,2)$ & 8 & $(3,5,1,2,4,6)$ ; $(5,3,6,1,2,4)$ & 14 &  &  & $(3,5,6,1,4,2)$ ; $(5,3,1,6,2,4)$ & 17 \\
8 & $(5,3,4,1,2,6)$ & 8 & $(3,5,1,2,6,4)$ ; $(3,5,6,1,2,4)$ & 14 &  &  & $(5,3,1,6,2,4)$ ; $(3,5,6,1,4,2)$ & 17 \\
9 &  &  & $(3,5,1,4,6,2)$ ; $(3,5,1,4,6,2)$ & 14 &  &  & $(5,6,1,3,2,4)$ ; $(3,5,4,1,6,2)$ & 17 \\
10 &  &  & $(3,5,1,4,6,2)$ ; $(5,1,3,6,2,4)$ & 14 &  &  & $(3,4,5,1,2,6)$ ; $(5,6,3,4,1,2)$ & 18 \\
11 &  &  & $(3,5,1,6,2,4)$ ; $(3,5,1,6,2,4)$ & 14 &  &  & $(3,5,4,1,2,6)$ ; $(5,3,6,4,1,2)$ & 18 \\
12 &  &  & $(3,5,4,6,1,2)$ ; $(5,1,2,3,6,4)$ & 14 &  &  & $(5,1,2,6,3,4)$ ; $(5,6,3,4,1,2)$ & 18 \\
13 &  &  & $(3,5,6,1,2,4)$ ; $(3,5,1,2,6,4)$ & 14 &  &  & $(5,1,6,3,4,2)$ ; $(5,1,6,3,4,2)$ & 18 \\
14 &  &  & $(5,1,2,3,4,6)$ ; $(5,3,4,6,1,2)$ & 14 &  &  & $(5,3,1,4,2,6)$ ; $(5,6,3,1,4,2)$ & 18 \\
15 &  &  & $(5,1,2,3,6,4)$ ; $(3,5,4,6,1,2)$ & 14 &  &  & $(5,3,1,6,4,2)$ ; $(5,3,1,6,4,2)$ & 18 \\
16 &  &  & $(5,1,2,6,3,4)$ ; $(3,4,5,6,1,2)$ & 14 &  &  & $(5,3,4,1,2,6)$ ; $(5,3,4,6,1,2)$ & 18 \\
17 &  &  & $(5,1,2,6,3,4)$ ; $(5,6,1,2,3,4)$ & 14 &  &  & $(5,3,4,1,6,2)$ ; $(5,3,4,1,6,2)$ & 18 \\
18 &  &  & $(5,1,3,2,4,6)$ ; $(5,6,1,3,2,4)$ & 14 &  &  & $(5,3,4,6,1,2)$ ; $(5,3,4,1,2,6)$ & 18 \\
19 &  &  & $(5,1,3,4,6,2)$ ; $(5,1,3,4,6,2)$ & 14 &  &  & $(5,3,6,4,1,2)$ ; $(3,5,4,1,2,6)$ & 18 \\
20 &  &  & $(5,1,3,6,2,4)$ ; $(3,5,1,4,6,2)$ & 14 &  &  & $(5,6,3,1,4,2)$ ; $(5,3,1,4,2,6)$ & 18 \\
21 &  &  & $(5,1,3,6,2,4)$ ; $(5,1,3,6,2,4)$ & 14 &  &  & $(5,6,3,4,1,2)$ ; $(3,4,5,1,2,6)$ & 18 \\
22 &  &  & $(5,1,6,2,3,4)$ ; $(3,4,5,1,6,2)$ & 14 &  &  & $(5,6,3,4,1,2)$ ; $(5,1,2,6,3,4)$ & 18 \\
23 &  &  & $(5,1,6,2,3,4)$ ; $(5,1,6,2,3,4)$ & 14 &  &  &  &  \\
24 &  &  & $(5,3,4,6,1,2)$ ; $(5,1,2,3,4,6)$ & 14 &  &  &  &  \\
25 &  &  & $(5,3,6,1,2,4)$ ; $(3,5,1,2,4,6)$ & 14 &  &  &  &  \\
26 &  &  & $(5,6,1,2,3,4)$ ; $(3,4,5,1,2,6)$ & 14 &  &  &  &  \\
27 &  &  & $(5,6,1,2,3,4)$ ; $(5,1,2,6,3,4)$ & 14 &  &  &  &  \\
28 &  &  & $(5,6,1,3,2,4)$ ; $(5,1,3,2,4,6)$ & 14 &  &  &  &  \\
29 &  &  & $(1,5,2,3,6,4)$ ; $(5,6,3,1,4,2)$ & 15 &  &  &  &  \\
30 &  &  & $(1,5,6,3,2,4)$ ; $(5,1,3,6,4,2)$ & 15 &  &  &  &  \\
31 &  &  & $(3,5,4,1,6,2)$ ; $(5,3,1,4,2,6)$ & 15 &  &  &  &  \\
32 &  &  & $(5,1,3,6,4,2)$ ; $(1,5,6,3,2,4)$ & 15 &  &  &  &  \\
33 &  &  & $(5,3,1,4,2,6)$ ; $(3,5,4,1,6,2)$ & 15 &  &  &  &  \\
34 &  &  & $(5,6,3,1,4,2)$ ; $(1,5,2,3,6,4)$ & 15 &  &  &  &  \\
35 &  &  & $(1,2,5,6,3,4)$ ; $(5,6,3,4,1,2)$ & 16 &  &  &  &  \\
36 &  &  & $(1,5,6,3,4,2)$ ; $(1,5,6,3,4,2)$ & 16 &  &  &  &  \\
37 &  &  & $(3,4,1,2,5,6)$ ; $(5,6,3,4,1,2)$ & 16 &  &  &  &  \\
38 &  &  & $(5,3,4,1,2,6)$ ; $(5,3,4,1,2,6)$ & 16 &  &  &  &  \\
39 &  &  & $(5,6,3,4,1,2)$ ; $(1,2,5,6,3,4)$ & 16 &  &  &  &  \\
40 &  &  & $(5,6,3,4,1,2)$ ; $(3,4,1,2,5,6)$ & 16 &  &  &  &  \\
\bottomrule
\end{longtable}
\end{center}\small
\underline{$k=3$:} $L_2\cong {\rm Res}_{E/F}(\GL_3)$. 
\begin{center}
\tiny
\setlength{\tabcolsep}{2pt}\renewcommand{\arraystretch}{1.06}
\begin{longtable}{r p{1.78cm} c p{4cm} c | p{1.78cm} c p{4.15cm} c}
\toprule
& \multicolumn{4}{c|}{$s_{w_\bullet}=\frac12$} & \multicolumn{4}{c}{$s_{w_\bullet}=1$}\\
\midrule
& \(w_{\vv_{\mathbb R}}\)  & $\ell(w_{\vv_{\mathbb R}} )$ & \quad $w_{\vv_{\mathbb C}}=(w_+, w_-)$ & $\ell(w_{\vv_{\mathbb C}})$ & \ \(w_{\vv_{\mathbb R}}\)  & $\ell(w_{\vv_{\mathbb R}} )$ & \quad $w_{\vv_{\mathbb C}}=(w_+, w_-)$ & \(\ell(w_{v_{\mathbb C}})\)\\
\midrule
\endhead
1 & $(4,1,5,2,6,3)$ & 6 & $(4,1,2,3,5,6)$ ; $(4,5,6,1,2,3)$ & 12 &  \ \  $\emptyset$ &  \  $\emptyset$& $(4,5,1,2,3,6)$ ; $(4,5,6,1,2,3)$ & 15 \\
2 & $(4,5,1,2,3,6)$ & 6 & $(4,1,2,5,3,6)$ ; $(4,5,1,6,2,3)$ & 12 &  &  & $(4,5,1,2,6,3)$ ; $(4,5,1,6,2,3)$ & 15 \\
3 &  &  & $(4,1,2,5,6,3)$ ; $(4,1,5,6,2,3)$ & 12 &  &  & $(4,5,1,6,2,3)$ ; $(4,5,1,2,6,3)$ & 15 \\
4 &  &  & $(4,1,5,2,3,6)$ ; $(4,5,1,2,6,3)$ & 12 &  &  & $(4,5,6,1,2,3)$ ; $(4,5,1,2,3,6)$ & 15 \\
5 &  &  & $(4,1,5,2,6,3)$ ; $(4,1,5,2,6,3)$ & 12 &  &  &  &  \\
6 &  &  & $(4,1,5,6,2,3)$ ; $(4,1,2,5,6,3)$ & 12 &  &  &  &  \\
7 &  &  & $(4,5,1,2,3,6)$ ; $(4,5,1,2,3,6)$ & 12 &  &  &  &  \\
8 &  &  & $(4,5,1,2,6,3)$ ; $(4,1,5,2,3,6)$ & 12 &  &  &  &  \\
9 &  &  & $(4,5,1,6,2,3)$ ; $(4,1,2,5,3,6)$ & 12 &  &  &  &  \\
10 &  &  & $(4,5,6,1,2,3)$ ; $(4,1,2,3,5,6)$ & 12 &  &  &  &  \\
\bottomrule
\end{longtable}
\end{center} 
\noindent Using \eqref{sw}, we summarize the above in the following table:
\small
\begin{table}[htbp]
\raggedright
\begin{tabular}{c c c c}
\toprule
max.\ $F$-parab. & $s_w$ & \#$w$'s & $\ell(w)$ $(\#w\text{'s of length }\ell(w))$\\
\midrule
$P_1$ & $\tfrac12$ & 33 & $15$ $(\#21)$, $16$ $(\#12)$ \\
& $1$ & 0 & $\varnothing$ \\ 
\midrule
$P_2$ & $\tfrac12$ & 320 & $21$ $(\#168)$, $22$ $(\#92)$, $23$ $(\#48)$, $24$ $(\#12)$ \\
& $1$ & 132 & $24$ $(\#15)$, $25$ $(\#27)$, $26$ $(\#51)$, $27$ $(\#39)$ \\ 
\midrule
$P_3$ & $\tfrac12$ & 20 & $18$ $(\#20)$ \\
& $1$ & 0 & $\varnothing$ \\
\bottomrule
\end{tabular}
\caption{\small Kostant representatives giving rise to the evaluation points \(s_w=\tfrac12\) and \(s_w=1\).}
\label{tab:kostant-representatives-evaluation-points}
\end{table}

\FloatBarrier
\normalsize

\section{Automorphic Cohomology}

\subsection{A few generalities about residual classes}\label{sect:residualclasses}

Applying the functor of $(\g,K_G)$-cohomology (or, equivalently, $(\g,\k_\infty)$-cohomology) to \eqref{eq:parabsuppdec} and \eqref{eq:cuspsuppdec}, induces an algebraic direct sum decomposition as $G(\A_f)$-module
\begin{equation}\label{eq:cohdec}
	H^q(\g,K_G, \mathcal{A}^\infty_{\J}(G)\otimes \EE_\mu) \cong \bigoplus_{\{P\}}\bigoplus_{\varphi_P} H^q(\g,K_G, \mathcal A^\infty_{\mathcal J,\{P\},\varphi_P}(G)\otimes\EE_\mu)
\end{equation}
which just amounts to a ``smooth-automorphic'' rephrasing of \cite{schwfr}, Thm.\ 2.3, which is the original source of this result. Decompose the discrete spectrum of $G(\A_F)$ as an orthogonal direct sum
$$L^2_{dis}(G(F)\backslash G(\A_F))=L^2_{cusp}(G(F)\backslash G(\A_F))\oplus_{\perp} L^2_{res}(G(F)\backslash G(\A_F))$$
and put $\mathcal{A}^\infty_{\J,res}(G):=\mathcal A^\infty_{\mathcal J}(G)\cap L^2_{res}(G(F)\backslash G(\A_F))$. A {\it residual automorphic cohomology class} is by definition a class in the image of the natural $G(\A_f)$-homomorphism in cohomology 
\begin{equation}\label{eq:rescoh}
	r^q: H^q(\g,K_G, \mathcal{A}^\infty_{\J,res}(G)\otimes \EE_\mu) \longrightarrow H^q(\g,K_G, \mathcal{A}^\infty_{\J}(G)\otimes \EE_\mu)
\end{equation}
given by the inclusion $\mathcal{A}^\infty_{\J,res}(G)\hookrightarrow\mathcal{A}^\infty_{\J}(G)$. Since $\mathcal A^\infty_{\mathcal J,\{G\}}(G)=\mathcal A^\infty_{\mathcal J}(G)\cap L^2_{cusp}(G(F)\backslash G(\A_F))$, cf.\, e.g., \cite{grob_book}, Lem.\ 15.1, residual automorphic cohomology is contained in 
\begin{equation}\label{eq:nonp}
\bigoplus_{\{P\}\neq \{G\}}\bigoplus_{\varphi_P} H^q(\g,K_G, \mathcal A^\infty_{\mathcal J,\{P\},\varphi_P}(G)\otimes\EE_\mu).
\end{equation}
In fact, setting 
$$\mathcal A^\infty_{res,\varphi_P}(G):=\mathcal{A}^\infty_{\J,res}(G)\cap\mathcal A^\infty_{\mathcal J,\{P\},\varphi_P}(G),$$ 
then the decompositions \eqref{eq:parabsuppdec} and \eqref{eq:cuspsuppdec} induce a decomposition
\begin{equation}\label{eq:nonp}
\mathcal{A}^\infty_{\J,res}(G)\cong \bigtoplus_{\{P\}\neq \{G\}}\bigtoplus_{\varphi_P} \mathcal A^\infty_{res,\varphi_P}(G).
\end{equation}
Each summand on the right hand side is a direct sum
\begin{equation}\label{eq:irrdecA} 
\mathcal A^\infty_{res,\varphi_P}(G)\cong \bigtoplus_{i\in I_{\varphi_P}} \mathcal H^{\infty_\A}_i
\end{equation}
stemming from partitioning the countable indexing set $I_{res}$ in the decomposition
$$L^2_{res}(G(F)\backslash G(\A_F))\cong \widehat{\bigoplus_{i\in I_{res}}}\ \mathcal H_i$$
into a direct Hilbert sum over a (suitable) countable family of irreducible unitary subrepresentations of $L^2_{dis}(G(F)\backslash G(\A_F))$, accordingly: $I_{res}=\bigsqcup_{\{P\}\neq \{G\}, \varphi_P} I_{\varphi_P}.$ Then, each space $\mathcal A^\infty_{res,\varphi_P}(G)$ consists of irreducible smooth $G(\A_F)$-representations on $LF$-spaces, which are all nearly equivalent, i.e., are all contained in the same near equivalence class of irreducible smooth-automorphic representations of $G(\A_F)$. 
If Arthur's Multiplicity Formula holds true for the discrete automorphic spectrum of $G$, then none of these irreducible $G(\A_F)$-representations $\mathcal H^{\infty_\A}_i$ will appear more than once in $\mathcal{A}^\infty_{\J,res}(G)$ (whence in each $\mathcal A^\infty_{res,\varphi_P}(G)$), i.e., already the equivalence class of each $G(\A_F)$-representation $\mathcal H^{\infty_\A}_i$ would determine a unique cuspidal support $\varphi_P=\varphi_P([\mathcal H^{\infty_\A}_i])$. See also \cite{grob_book}, Lem.\ 15.10. (We do not assume AMF, however.)\\\\
In summary, we obtain algebraic direct sum decompositions as $G(\A_f)$-modules
$$H^q(\g,K_G, \mathcal{A}^\infty_{\J,res}(G)\otimes \EE_\mu) \cong \bigoplus_{\{P\}\neq \{G\}} \bigoplus_{\varphi_P}H^q(\g,K_G, \mathcal A^\infty_{res,\varphi_P}(G)\otimes \EE_\mu) \cong  \bigoplus_{\{P\}\neq \{G\}} \bigoplus_{\varphi_P}\bigoplus_{i\in I_{\varphi_P}} H^q(\g,K_G, \mathcal H^{\infty_\A}_i\otimes \EE_\mu),$$
which allows one to write $r^q$ as a sum of restrictions $r^q=\oplus_{\{P\}\neq \{G\}, \varphi_P}r^q_{\varphi_P}=\oplus_{\{P\}\neq \{G\},\varphi_P, i\in I_{\varphi_P}} r^q_i$, where for every $\varphi_P$ and $i\in I_{\varphi_P}$ the respective $G(\A_f)$-morphisms are defined / fit into the following commuting diagram
$$\xymatrix{
H^{q}(\g, K_G, \mathcal A^\infty_{\J,res}(G)\otimes\EE_\mu) \ar[rrr]^{r^q}  & & &
H^q(\g,K_G, \mathcal A^\infty_{\mathcal J}(G)\otimes\EE_\mu)\\ &&&\\
H^{q}(\g, K_G, \mathcal A^\infty_{res,\varphi_P}(G)\otimes\EE_\mu) \ar@{^{(}->}[uu]^{}\ar[rrr]^{r^q_{\varphi_P}}  & & &
H^q(\g,K_G, \mathcal A^\infty_{\mathcal J,\{P\},\varphi_P}(G)\otimes\EE_\mu)\ar@{^{(}->}[uu]^{}\\ &&&\\
H^q(\g,K_G, \mathcal H^{\infty_\A}_i\otimes \EE_\mu)\ar@{^{(}->}[uu]^{}
\ar@{->}[uurrr]^{r^q_i}
 & & &
}$$
As we just observed above, the targeted cuspidal datum $\varphi_P$ is furthermore uniquely determined by the isomorphism class $[\mathcal H^{\infty_\A}_i]$ of $\mathcal H^{\infty_\A}_i$, if AMF holds.

\subsection{The main result on residual Eisenstein cohomology of $G_n(\A_F)$}
The maximal parabolic $F$-subgroups $P_k$, $1\leq k\leq r$, of $G$ and their associate classes of cohomological cuspidal automorphic representations $\varphi_{P_k}$ obviously constitute direct admissible $G(\A_f)$-submodules $H^q(\g,K_G, \mathcal A^\infty_{\mathcal J,\{P_k\},\varphi_{P_k}}(G)\otimes\EE_\mu)$ of automorphic cohomology as well as direct $G(\A_F)$-summands $\mathcal A^\infty_{res,\varphi_{P_k}}(G)$ of \eqref{eq:nonp}.\\\\ The following is our main theorem, which treats the existence of families of non-trivial residual automorphic cohomology classes in the images of $r^q_{\varphi_{P_k}}$. For sake of readability we recall the various assumptions made so far in one place:

\begin{thm}\label{thm:main}
Let $E/F$ be a quadratic extension of arbitrary number fields and let $G=G_n$ be the unitary group over $F$, attached to a non-dengenerate hermitian form on an $n$-dimensional $E$-vector space of $F$-rank $r\geq 1$. Let $P_k$, $1\leq k\leq r$, be a standard maximal parabolic $F$-subgroup of $G$ and let $\varphi_{P_k} $ be an associate class of cuspidal automorphic representations of $L_k(\A_F)\cong \GL_k(\A_E)\times G_{n-2k}(\A_F)$, cf.\ Sect.\ \ref{sect:pc}, which is represented by a cuspidal automorphic representation $\pi_{\lambda_0}=\pi\otimes e^{\langle \lambda_0, H_P(\cdot)\rangle}$. We write, as we may, $\pi=\tau\hat\otimes\sigma$ for a unitary cuspidal automorphic representation $\tau$ of $\GL_k(\A_E)$ and a unitary cohomological cuspidal automorphic representation $\sigma$ of $G_{n-2k}(\A_F)$ and we let  $w\in W^{P_k}$ be the corresponding unique Kostant representative such that $\pi_\infty$ is $(\m_{k,\infty},K_{M^\circ_{k,\infty}})$-cohomological with respect to $\EE_{\mu_w}|_{M^\circ_{k,\infty}}$ and such that $\lambda_0=s_w\cdot \gamma_k$ with $s_w$ being given by \eqref{sw}. Furthermore, we abbreviate
$$q(\pi_{\lambda_0}):=q_{min}(\pi_\infty) + [F:\Q] k(2n-3k) -\ell(w),$$
where $q_{min}(\pi_\infty)$ is given by Prop.\ \ref{prop:qminLk}.
We assume that $\Pi=BC(\sigma)$ satisfies the equivalent conditions of Thm.\ \ref{thm:qtcusp}, i.e., $\Pi=\Pi_1\boxplus...\boxplus\Pi_l$ for cuspidal automorphic representations $\Pi_i$, $1\leq i\leq l$ and that either
\begin{enumerate}
\item $s_w=\tfrac12$ and $\tau$ is conjugate self-dual, $\epsilon$-distinguished and $L(\tfrac12,\tau\times \Pi^\vee)\neq 0$; or
\item $s_w=1$ and $\tau=\Pi_i$ for some $1\leq i\leq l$.
\end{enumerate}
Then, $r^{q(\pi_{\lambda_0})}_{\varphi_{P_k}}$ is injective and its image, which consists of residual Eisenstein cohomology classes, is non-zero.
\end{thm}
\begin{proof}
Let $\pi_{\lambda_0}=\pi\otimes e^{\langle \lambda_0, H_P(\cdot)\rangle}\in \varphi_{P_k}$ be as in the statement of the theorem. Then, there exist smooth $K_\infty$-finite sections $f\in I^G_{P_k}(\pi)$, such that the Eisenstein series $E_{P_k}(f,\lambda)$ have a simple pole at $\lambda_0=s_w\cdot \gamma_k$ by Thm.\ \ref{thm:Poles}. As $s_w>0$, it follows from Langlands's square-integrability criterion, \cite{moewal}, Lem.\ I.4.11 or \cite{langlandsLNM}, Cor.\ on p.\ 104, that the automorphic subrepresentations generated by the attached residues $\lim_{s\rightarrow s_w^+}(s-s_w)E_{P_k}(f,s\cdot\gamma_k)$ are indeed square-integrable and in fact their topological closure in $\mathcal A^\infty_\J(G)$ spans the non-zero space $\mathcal A^\infty_{res,\varphi_{P_k}}(G) = {\rm Cl}_{\mathcal A^\infty_{\mathcal J}(G)}({\rm Eis}_{\pi, s_w\cdot\gamma_k}(I^G_{P_k}(\pi)_{(K_\infty)}))$, see \eqref{eq:Eismap} and \eqref{eq:defcuspsup}. (Here, we used the identification $I^G_{P_k}(\pi)_{(K_\infty)}\hookrightarrow I^G_{P_k}(\pi)_{(K_\infty)}\otimes {S}(\check\a_{P_k,\C})$, $f\mapsto f\otimes 1$.)\\\\ 
Let now $\mathcal H^{\infty_\A}$ be an irreducible summand of $\mathcal A^\infty_{res,\varphi_{P_k}}(G)$, cf.\ \eqref{eq:irrdecA}, which we recall is derived from an irreducible summand $\mathcal H$ of $L^2_{res}(G(F)\backslash G(\A_F))$. The induced representation ${\rm Ind}_{P_{k,\infty}}^{G_\infty}[\pi_\infty\otimes e^{\langle s_w \gamma_k, H_{P_\infty}(\cdot)\rangle}]$ is in Langlands position, as $\pi_\infty$ is tempered by Lem.\ \ref{lem:varphis} and Thm.\ \ref{thm:qtcusp}, and $s_w\gamma_k\in \check\a_{P_k}^{+}$, as $s_w>0$, and therefore has a unique irreducible quotient $J(\pi_\infty,s_w)$. Via the constant term map along $P_k$ it is infinitesimally equivalent to the unitary representation $\mathcal H_\infty$, the archimedean component of $\mathcal H$. Consequently, the infinitesimal character of $\mathcal H_\infty$ is the one of ${\rm Ind}_{P_{k,\infty}}^{G_\infty}[\pi_\infty\otimes e^{\langle s_w \gamma_k, H_{P_\infty}(\cdot)\rangle}]$, whence equal to the one of $\EE^\vee_\mu$. It hence follows now along the same lines, as carried out in details in the proof of Lem.\ \ref{lem:varphis},  -- that is, by applying \cite{salamribsu}, Thm.\ 1.8, to the restriction of the irreducible unitary summands of $\mathcal H_\infty|_{G_\infty^\circ}$ to the connected semisimple derived group $\mathcal D(G^\circ_\infty)\cong SU(V_n)_\infty$ -- that $\mathcal H_\infty$ is cohomological with respect to $\EE_\mu$. Knowing that $\mathcal H_\infty$ is infinitesimally equivalent to $J(\pi_\infty,s_w)$, its minimal degree of non-vanishing $(\g,K_G)$-cohomomology may be computed by several general techniques, e.g., applying \cite{vozu}, Thm.\ 6.16 and Thm.\ 5.5, {\it ibidem}, or, which seems most convenient for our purposes, using \cite{CKT}, Thm.\ 9.2 and \cite{bowa}, Thm.\ III.3.3 and V.1.4 -- V.1.5, {\it ibidem}, to finally obtain 
\begin{equation}\label{eq:CKTinj}
q_{min}(\mathcal H_\infty)=q_{min}(\pi_\infty) + (\dim_{\R}(N_{k,\infty})-\ell(w))=q_{min}(\pi_\infty) + [F:\Q] k(2n-3k) -\ell(w)=q(\pi_{\lambda_0}).
\end{equation}
Therefore, $H^{q(\pi_{\lambda_0})}(\g,K_G, \mathcal H^{\infty_\A}\otimes \EE_\mu)$ is non-zero for each irreducible $G(\A_F)$-summand $ \mathcal H^{\infty_\A}$ of $\mathcal A^\infty_{res,\varphi_{P_k}}(G)$ and so, under the assumptions of the theorem,
$$H^{q(\pi_{\lambda_0})}(\g,K_G, \mathcal A^\infty_{res,\varphi_{P_k}}(G)\otimes\EE_\mu)\neq \{0\}.$$ 
We are hence left to show that $r^{q(\pi_{\lambda_0})}_{\varphi_{P_k}}$ is injective to complete the proof.\\\\ 
Since $r^{q(\pi_{\lambda_0})}_{\varphi_{P_k}}$ is a $G(\A_f)$-homomorphism, it will be shown to be injective, once it is known to be injective on the isotypic components of the semisimple $G(\A_f)$-module $H^{q(\pi_{\lambda_0})}(\g,K_G, \mathcal A^\infty_{res,\varphi_{P_k}}(G)\otimes\EE_\mu)$. However, if $\ell(w)>\sum_{\vv\in S_\infty}\lfloor\tfrac12 \dim_{\R}(N_{k}(F_\vv))\rfloor$, the latter injectivity on the $G(\A_f)$-sotypic components is, by \eqref{eq:CKTinj}, a direct consequence of \cite{grobner-EisRes}, Cor.\ 17, which was proved using Franke's filtration. We remark that the condition $\ell(w)>\sum_{\vv\in S_\infty}\lfloor\tfrac12 \dim_{\R}(N_{k}(F_\vv))\rfloor$ holds in the majority of cases of $w\in W^{P_k}$ -- see Table \ref{tab:kostant-representatives-evaluation-points} for concrete examples -- as one has $\ell(w)\geq \sum_{\vv\in S_\infty}\lceil\tfrac12 \dim_{\R}(N_{k}(F_\vv))\rceil$ in general, cf.\ \cite{grobner-EisRes}, Prop.\ 12.\\\\ 
In order to cover also the remaining case, when $\ell(w)=\sum_{\vv\in S_\infty}\lfloor\tfrac12 \dim_{\R}(N_{k}(F_\vv))\rfloor $ (which, as we also remark, may indeed happen for both cases, i.e., when $s_w=\tfrac12$ and when $s_w=1$), we adopt the ideas of the proof of Thm.\ 2.1 in \cite{grob22}, in order to show the injectivity of  $r^{q(\pi_{\lambda_0})}_{\varphi_{P_k}}$: Fix an irreducible summand $\mathcal H^{\infty_\A}$ occurring in $\mathcal A^\infty_{res,\varphi_{P_k}}(G)$ with (necessarily finite) multiplicity $m(\mathcal H)$ and let $\mathcal V_{\mathcal H^{\infty_\A}}$ denote the corresponding $\mathcal H^{\infty_\A}$-isotypic component in $\mathcal A^\infty_{res,\varphi_{P_k}}(G)$. We first describe the constant term along $P_k$ on $\mathcal V_{\mathcal H^{\infty_\A}}$. Recall from Sect.\ \ref{sect:intops} (our choice of a representative of) the unique non-trivial element $w_k$ of $W_k$, which conjugates $P_k$ to its opposite parabolic. Since the residual pole at \(\lambda_0\) occurs only for a conjugate self-dual datum, the assumptions of the theorem imply that ${w_k}(\pi)\cong\pi$ and hence ${\rm Ind}_{P_k(\A_F)}^{G(\A_F)}[\pi\otimes e^{\langle -s_w\gamma_k,H_{P_k}(\cdot)\rangle}]$ is nothing but the anti-standard induced module of ${\rm Ind}_{P_k(\A_F)}^{G(\A_F)}[\pi\otimes e^{\langle s_w\gamma_k,H_{P_k}(\cdot)\rangle}]$. We observe that the constant term along $P_k$ takes
values in ${\rm Ind}_{P_k(\A_F)}^{G(\A_F)}[\pi^{m(\pi)}\otimes e^{\langle -s_w\gamma_k,H_{P_k}(\cdot)\rangle}]$: Indeed, if
$R(f):=\lim_{s\rightarrow s_w^+}(s-s_w)E_{P_k}(f,s\gamma_k)$ is one of the residues generating $\mathcal A^\infty_{res,\varphi_{P_k}}(G)$, then the constant-term
formula along $P_k$, cf.\ \cite{moewal}, Prop.\ II.1.7, gives, after exchanging integration over $N_k(F)\backslash N_k(\A_F)$ with taking the limit $\lim_{s\rightarrow s_w^+}$, 
\begin{equation}\label{eq:constant-term-residue}
 (R(f))_{P_k}
 =
 {\rm Res}_{s=s_w}\,M(s,\pi)(f\cdot e^{\<s\gamma_k+\rho_{P_k},H_{P_k}(\cdot)\>}),
\end{equation}
as the summand $f\cdot e^{\<s\gamma_k+\rho_{P_k},H_{P_k}(\cdot)\>}$ attached to the identity-intertwiner of the constant term contributes trivially, since it
is holomorphic at $s=s_w$. Consequently, computing the constant term along $P_k$ induces a $G(\A_F)$-equivariant linear map
$$
 c_{\mathcal H^{\infty_\A}}:\mathcal V_{\mathcal H^{\infty_\A}}
 \longrightarrow
 {\rm Ind}_{P_k(\A_F)}^{G(\A_F)}[\pi^{m(\pi)}\otimes e^{\langle -s_w\gamma_k,H_{P_k}(\cdot)\rangle}].
$$
This map is injective. In fact, if
$\phi\in\mathcal V_{\mathcal H^{\infty_\A}}$ satisfies $c_{\mathcal H^{\infty_\A}}(\phi)=0$, then the preceding observation implies that its entire constant term $\phi_{P_k}$ vanishes. Since $P_k$ is self-associate, the vanishing of $\phi_{P_k}$ therefore implies that $\phi_{P_j}=0$ for all maximal parabolic subgroups, $1\leq j\leq r$, see \cite{moewal}, Prop.\ II.1.7 again. Hence, $\phi$ is also cuspidal, implying $\phi=0$.\\\\ 
By \cite{grob_book}, Thm.\ 12.16 (see also \cite{grob_zun}, Thm.\ 3.15) and using the identification $\mathcal H^{\infty_\A}_\infty\cong J(\pi_\infty,s_w)$ obtained above, the full smooth isotypic component admits a realization $\mathcal V_{\mathcal H^{\infty_\A}}\cong J(\pi_\infty,s_w)\,\overline\otimes_{\rm in}\,\mathcal (H^{\infty_f}_f)^{m(\mathcal H)}$.
Let
$$c_\infty:J(\pi_\infty,s_w)\hookrightarrow {\rm Ind}_{P_{k,\infty}}^{G_\infty}[\pi_\infty\otimes e^{\langle -s_w \gamma_k, H_{P_\infty}(\cdot)\rangle}]$$ 
denote the canonical embedding of $G_\infty$-representations, the $G(\A_F)$-equivariance of
$c_{\mathcal H^{\infty_\A}}$ therefore implies $$
 c_{\mathcal H^{\infty_\A}}
 =c_\infty\overline\otimes_{\rm in}c_f
$$
for a uniquely determined $G(\A_f)$-homomorphism
$$
 c_f:
 (\mathcal H^{\infty_f}_f)^{m(\mathcal H)} \longrightarrow
  {\rm Ind}_{P_k(\A_f)}^{G(\A_f)}[\pi_f^{m(\pi)}\otimes e^{\langle -s_w\gamma_k,H_{P_{k,f}}(\cdot)\rangle}].
 $$
Since $c_{\mathcal H^{\infty_\A}}$ and $c_\infty$ are injective, $c_f$ is injective as well.\\\\
We are finally ready to look at the map induced from $c_{\mathcal H^{\infty_\A}}$ in $(\g,K_G)$-cohomology in degree $q(\pi_{\lambda_0})$,
\begin{equation*}
c_{\mathcal H^{\infty_\A}}^{q(\pi_{\lambda_0})}: H^{q(\pi_{\lambda_0})}(\g,K_G, \mathcal V_{\mathcal H^{\infty_\A}}\otimes \EE_\mu) \longrightarrow H^{q(\pi_{\lambda_0})}(\g,K_G, {\rm Ind}_{P_k(\A_F)}^{G(\A_F)}[\pi^{m(\pi)}\otimes e^{\langle -s_w\gamma_k,H_{P_k}(\cdot)\rangle}]\otimes\EE_\mu).
\end{equation*}
However, this map is just the one, which one obtains from composing the restriction of $r^{q(\pi_{\lambda_0})}_{\varphi_{P_k}}$ to the $G(\A_f)$-isotypic summand $H^{q(\pi_{\lambda_0})}(\g,K_G, \mathcal V_{\mathcal H^{\infty_\A}}\otimes \EE_\mu)$ of $H^{q(\pi_{\lambda_0})}(\g,K_G, \mathcal A^\infty_{res,\varphi_{P_k}}(G)\otimes\EE_\mu)$ with the map in $(\g,K_G)$-cohomology, induced from computing the constant term on $\mathcal A^\infty_{\mathcal J,\{P_k\},\varphi_{P_k}}(G)$ along $P_k$ and projecting onto its $w_k$-summand. Hence, in order to show that $r^{q(\pi_{\lambda_0})}_{\varphi_{P_k}}$ is injective, it is enough to show that $c_{\mathcal H^{\infty_\A}}^{q(\pi_{\lambda_0})}$ is. Clearly, since $c_{\mathcal H^{\infty_\A}}^{q(\pi_{\lambda_0})}=c_\infty^{q(\pi_{\lambda_0})}\otimes c_f$, where
$$c_\infty^{q(\pi_{\lambda_0})}: H^{q(\pi_{\lambda_0})}(\g,K_G, J(\pi_\infty,s_w)\otimes \EE_\mu) \longrightarrow H^{q(\pi_{\lambda_0})}(\g,K_G, {\rm Ind}_{P_{k,\infty}}^{G_\infty}[\pi_\infty\otimes e^{\langle -s_w\gamma_k, H_{P_{k,\infty}}(\cdot)\rangle}]\otimes\EE_\mu)$$
is the map in $(\g,K_G)$-cohomology, induced from $c_\infty$, and since we have just seen that $c_f$ is injective, the global map $c_{\mathcal H^{\infty_\A}}^{q(\pi_{\lambda_0})}$ is injective, if and only if $c_\infty^{q(\pi_{\lambda_0})}$. However, the desired injectivity of $c_\infty^{q(\pi_{\lambda_0})}$ -- a purely archimedean assertion -- was achieved in the generality needed here in the aforementioned \cite{CKT}, Thm.\ 9.2, paraphrasing and extending results of \cite{rosp}. This completes the proof of the theorem.

\end{proof}

\begin{rem}[Non-zero restriction to boundary strata]
Variants of the part in the above argument for the injectivity of $r^{q(\pi_{\lambda_0})}_{\varphi_{P_k}}$, when $\ell(w)=\sum_{\vv\in S_\infty}\lfloor\tfrac12 \dim_{\R}(N_{k}(F_\vv))\rfloor $, taken from \cite{grob22}, may also be found in the proof of \cite{rosp}, Thm.\ III.1 and of \cite{nairrai}, Lem.\ B.1. As a byproduct, it shows that the restriction of the non-zero residual Eisenstein cohomology classes in the image of $r^{q(\pi_{\lambda_0})}_{\varphi_{P_k}}$ to the boundary stratum $\partial_{P_k}(\overline X_G)$ in the Borel-Serre compactification of $X_G:=G_n(F)\backslash G_n(\A_F)/K_G$ attached to $P_k$ (\cite{boserre}; \cite{rohlfs}, \S 6.4) are non-zero as well. Consequently, the image of $r^{q(\pi_{\lambda_0})}_{\varphi_{P_k}}$ has trivial intersection with interior cohomology, i.e., none of the non-zero residual Eisenstein cohomology classes constructed in Thm.\ \ref{thm:main} is compactly supported. (Consequently, none of them can be a {\it ghost class}.) See also \cite{grobner-EisRes}, \S8, for complementary results.
\end{rem}

\begin{rem}[Complementary vanishing results]
In the special case when $E/\Q$ is an imaginary quadratic number field, $\mathcal D(G(\R))\cong SU(p,q)$, $k=1$ and $s_w=\tfrac12$, the injectivity of $r^{q(\pi_{\lambda_0})}_{\varphi_{P_k}}=r^{pq-1}_{\varphi_{P_1}}$ was also shown recently in \cite{mundy}. His main result, however, also implies the {\it vanishing} of the map $r^{pq+1}_{\varphi_{P_1}}$ (hence slightly sharpening Thm.\ 4.3.B of \cite{grabschwun}: Their ``{\it quotient in degree $2n-r+1$}'' is indeed trivial), a complementary question which we left completely untouched here, as its underlying techniques would go beyond the scope of this paper, but which we plan to take on in joint work with Mundy in the case of semisimple groups over totally real fields. 
\end{rem}

\subsection{A non-trivial example - Part II: Concrete non-zero residual Eisenstein classes}

We resume the discussion of our extended example $G=U(V_6,h)/\Q(\sqrt[3]{2})$ from Sect.\ \ref{sect:extex} in the light of Thm.\ \ref{thm:main}. Our aim is to give concrete examples of non-zero residual Eisenstein cohomolgy classes for $G(\A_F)$, by making all ingredients of our general theorem Thm.\ \ref{thm:main} explicit (and showing their existence). Firstly, we note that
$$
\dim_\mathbb R(G_\infty/K_G)
=
\dim_\mathbb R U(3,3)/ (U(3)\times U(3))
+
\dim_\mathbb R \GL_6(\mathbb C)/U(6)
=
2\cdot 3\cdot 3+6^2
=
54.
$$
and we summarize the calculated values of several input-data in the following table:
\begin{table}[ht]
\centering
\small
\renewcommand{\arraystretch}{1.2}
\begin{tabular}{c|c|c|c|c|c|c}
\hline
\(P_k\)
&
\(s_w\)
&
\(\begin{array}{c} q_{min}(\pi_\infty) \end{array}\)
&
\(\begin{array}{c} q(\pi_{\lambda_0}) \end{array}\)
&
\(\begin{array}{c} \dim_\mathbb R X_{L_k} \end{array}\)
&
\(\begin{array}{c} \dim_\mathbb R N_{k,\infty} \end{array}\)
&
\(\begin{array}{c} \dim_\mathbb R \partial_{P_k}(\overline X_G) \end{array}\)
\\
\hline
\(P_1\) & \(\frac12\) & \(10\) & \(21,\ 22\)             & \(26\) & \(27\) & \(53\)\\
\(P_2\) & \(\frac12\) & \(5\)  & \(17,\ 18,\ 19,\ 20\)   & \(17\) & \(36\) & \(53\)\\
\(P_2\) & \(1\)       & \(5\)  & \(14,\ 15,\ 16,\ 17\)   & \(17\) & \(36\) & \(53\)\\
\(P_3\) & \(\frac12\) & \(9\)  & \(18\)                  & \(26\) & \(27\) & \(53\)\\
\hline
\end{tabular}
\caption{Cohomological degrees and dimensions of related Borel--Serre boundary data}
\label{tab:example-degrees-boundary}
\end{table}

The main result of this subsection is the following theorem, which provides actual examples of non-zero resiudal Eisenstein cohomology classes for our role model case of the unitary group $G=U(V_6,h)/\Q(\sqrt[3]{2})$ from Sect.\ \ref{sect:extex}:

\begin{thm}\label{prop:nonempty-P2-example}
Let $w\in W^{P_2}$ be the Kostant representative $w=
        \bigl(
        w_{\vv_\R},w_{\vv_\C}
        \bigr)
        =
        \bigl(
        w_0,(w_0,w_0)
        \bigr)$, with $w_0=(5,3,1,6,4,2)$
occurring in the column of \(s_w=1\) in the second table of Sect.\ \ref{sect:tablesw}. Then there
exists a cuspidal automorphic representation $\pi=\tau\widehat\otimes\sigma$ of $L_2(\mathbb A_F)
        \cong
        \GL_2(\mathbb A_E)\times U(V_2,h)(\mathbb A_F)$
satisfying the hypotheses of Theorem~\ref{thm:main} at $\lambda_0=\gamma_2.$
Consequently, degree $ q(\pi_{\lambda_0})
        = 14$ in the third row of Table \ref{tab:example-degrees-boundary} carries non-zero residual Eisenstein cohomology classes. None of them lies in inner cohomology.
\end{thm}
\begin{proof}
By our recipe from Lem.\ \ref{lem:Kostant},  the pair $(I,J)$, given by
$$
        I=\{3,6\},
        \qquad
        J=\{1,4\}.
$$
determines $w_0$. We let $\sigma_\infty$ be a fixed cohomological tempered representation of
$$
        U(V_2,h)_\infty
        \cong
        U(1,1)\times \GL_2(\mathbb C)
$$
determined by the complement $\{2,5\}$ of this Kostant datum. More precisely, this means that, since $\mu=0$ in our example, we choose a discrete series representation of $U(1,1)$, which is cohomological with respect to the irreducible finite dimensional representation of highest weight $(1,-1)$ (for the very moment it is not relevant wether it is holomorphic or anti-holomorphic, which will become important only later) and form its tensor product with the (unique) tempered representation of $\GL_2(\C)$, which is cohomological with respect to the irreducible finite dimensional representation of highest weight $((1,-1),(1,-1))$. As above, we write $BC(\sigma_\infty)$ for its local archimedean base-change. Explicitly, it factors as the product 
\begin{equation}\label{BCok}
BC(\sigma_\infty)\cong \bigotimes_{\ww\in\Sigma_\infty}
        \left(z^{3/2}\bar z^{-3/2}\right)
        \times
        \left(z^{-3/2}\bar z^{3/2}\right)
\end{equation}
of three isomorphic fully induced representations of three copies of $\GL_2(\mathbb C)$. It follows from the particular choice of \(w\), that the representation 
$$
        \pi_\infty
        :=
        \tau_\infty\widehat\otimes\sigma_\infty
        \quad\text{with}\quad
        \tau_\infty:= BC(\sigma_\infty).
$$
is $(\m_{2,\infty}, K^\circ_{M_{2,\infty}})$-cohomological with respect to $\EE_{\mu_w}|_{M^\circ_{2,\infty}}$.\\\\ 
We now globalize $\tau_\infty$ and $\sigma_\infty$ to cuspidal automorphic representations $\tau$, respectively $\sigma$, which satisfy all the conditions of Thm.\ \ref{thm:main} for $\lambda_0=\gamma_2$. Once this is achieved, we obtain the desired conclusion of Thm.\ \ref{prop:nonempty-P2-example} by applying the main result Thm.\ \ref{thm:main} to the associate class $\varphi_{P_2}$ defined by $\pi_{\gamma_2}:=(\tau\hat\otimes\sigma)\otimes e^{\langle \gamma_2, H_{P_2}(\cdot)\rangle}$ and inspecting Table \ref{tab:example-degrees-boundary} for the explicit numerology concerning, e.g., $q(\pi_{\lambda_0})$.\\\\
To this end, let $L_1=\mathbb Q(\sqrt{-2})$. Denote by $c_1$ the only non-trivial automorphism of $L_1$ (i.e., complex conjugation). Let $\varepsilon_{L_1/\mathbb Q}: \mathbb Q^*\backslash\mathbb A_\mathbb Q^*\rightarrow\C^*$ be the quadratic Hecke character of attached to
\(L_1/\mathbb Q\) by class field theory and let $\eta_{L_1/\mathbb Q}: L_1^*\backslash \A_{L_1}^*\rightarrow\C^*$ by a ($c_1$-)conjugate self-dual extension of $\varepsilon_{L_1/\mathbb Q}$ as in Sect.\ \ref{sect:nf}. More concretely, one may pin down $\eta_{L_1/\mathbb Q}$ using an elliptic curve over
\(\mathbb Q\) with complex multiplication by \(\mathcal O_{L_1}\) and letting \(\psi\) be the algebraic Hecke character of \(L_1\) attached to such a CM elliptic curve by Deuring's construction, see, e.g., \cite[Chp.\ II, \S\S 9--10]{Sil94}. Then put $\eta_{L_1/\mathbb Q}
        :=
        \psi\,\|\cdot\|_{\mathbb A_{L_1}}^{-1/2}$. It follows that one has
$$
        \eta_{L_1/\mathbb Q,\infty}(z)
        =
        z^{1/2}\bar z^{-1/2}.
$$
Now define $\chi_{1,0}
        :=
        \eta_{L_1/\mathbb Q}^{\,3}.$ Then,
$$
        \chi_{1,0,\infty}(z)
        =
        z^{3/2}\bar z^{-3/2},
$$
while still
$$
        \chi_{1,0}
        \big|_{\mathbb A_\mathbb Q^\times}
        =
        \varepsilon_{L_1/\mathbb Q}^{\,3}
        =
        \varepsilon_{L_1/\mathbb Q},  \quad\quad {\rm and} \quad\quad   {}^{\,c_1} \chi_{1,0}=\chi_{1,0}^{-1}.
$$
We next modify \(\chi_{1,0}\) at one non-archimedean place, in order to obtain local regularity. Since \(L_1\) is linearly disjoint from the normal closure of \(E/\mathbb Q\), Chebotarev's theorem gives a prime number \(p\geq 3\), away from the finite conductor of \(\chi_{1,0}\), such that \(p\) splits completely in \(E\) and is inert in \(L_1\). Let $w_p$ be the unique place of \(L_1\) above \(p\) and choose a character
$$
        \beta:L_{1,w_p}^{*}\longrightarrow\mathbb C^{*}
$$
of odd order $b_\beta$ satisfying
$$
        \beta|_{\mathbb Q_p^*}=1 \quad\quad {\rm and}\quad\quad      \bigl(\chi_{1,0,w_p}\beta\bigr)^2\neq 1.
$$
Such a character exists because the unramified quotient $ L_{1,w_p}^{*}/\mathbb Q_p^*$ has arbitrarily ramified characters of \(p\)-power order. Since \(b_\beta\) is odd, choose
\(r\in\mathbb Z\) with $2r\equiv 1\pmod{b_\beta}$ and put
$$
        \alpha_{w_p}:=\beta^r.
$$
Then,
$$
        \alpha_{w_p}^2=\beta,
        \qquad
        \alpha_{w_p}|_{\mathbb Q_p^*}=1.
$$
By the Grunwald--Wang theorem (recall that $b_\beta$ and $p$ are odd), there exists hence a finite order Hecke character
$$
        \xi_0:L_1^*\backslash \mathbb A_{L_1}^{*}
        \longrightarrow \mathbb C^*
$$
such that
$$
        (\xi_0)_{w_p}=\alpha_{w_p} .
$$
See Artin--Tate, \cite{ArtinTate}, Chp.~X, Thm.~5, or the formulation in \cite{BCGNT}, Lem.\ 5.3.1.  Define
$$
        \xi:=\xi_0({}^{\,c_1}\xi_0)^{-1}.
$$
Then, $\xi$ is ($c_1$-)conjugate self-dual, i.e., ${}^{\,c_1}\xi=\xi^{-1}$, by construction and $\xi|_{\mathbb A_\mathbb Q^*}=1$, as \(c_1\) acts trivially on \(\mathbb A_\mathbb Q^*\). Moreover, at \(w_p\) we have $\xi_{w_p}=\alpha_{w_p}\,({}^{\,c_1}\alpha_{w_p})^{-1},$
which, since \(\alpha_{w_p}\) is trivial on \(\mathbb Q_p^*\) and therefore ${}^{\,c_1}\alpha_{w_p}=\alpha_{w_p}^{-1}$ (recall that $p$ is inert in $L_1$, whence $c_1$ is the non-trivial automorphism of $L_{1,w_p}/\Q_p$), simplifies to $\xi_{w_p} = \alpha_{w_p}^2 = \beta.$ Set
$$
        \chi_1:=\chi_{1,0}\cdot \xi.
$$
Then, $\chi_1$ is a ($c_1$-)conjugate self-dual Hecke character $L^*_1\backslash \A^*_{L_1}\rightarrow\C^*$ which satisfies
$$
        \chi_{1,\infty}(z)
        =
        z^{3/2}\bar z^{-3/2}, \quad\quad
        \chi_1|_{\mathbb A_\mathbb Q^*}
        =
        \varepsilon_{L_1/\mathbb Q},\quad\quad {\rm and} \quad\quad
        \chi_{1,w_p}^2\neq 1.
$$
Now put $L:=EL_1$ and let $\theta$ be the unique non-trivial element of \(\operatorname{Gal}(L/E)\). Thus, \(\theta\) acts trivially on \(E\) and as complex conjugation on \(L_1\). Let \(\widetilde c\) be the automorphism of \(L\) which acts as the non-trivial element \(c\in\operatorname{Gal}(E/F)\) on \(E\) and trivially on \(L_1\). Then \(\theta\) and \(\widetilde c\) commute. Define a Hecke character of \(L\) by
$$
        \chi:=\chi_1\circ N_{L/L_1}.
$$
Then, the following holds:
$$\boxed{{}^{\theta}\chi=\chi^{-1}\neq\chi \quad\quad {\rm and} \quad\quad {}^{\widetilde c}\chi=\chi}$$ 
Indeed, $N_{L/L_1}(\theta x)
        =
        c_1\bigl(N_{L/L_1}(x)\bigr)$ and since \({}^{c_1}\chi_1=\chi_1^{-1}\), it follows that ${}^\theta\chi=\chi^{-1}$, which shows the first claimed equation of the box. In order to see the inequality, mentioned there, we choose a place \(\ww_p\) of \(E\) above \(p\) and
a place \(u_p\) of \(L\) above \(\ww_p\). Since \(p\) splits completely in \(E\), one has
$$
        E_{\ww_p}\cong \mathbb Q_p,
        \qquad
        L_{u_p} \cong L_{1,w_p},
$$
whence under this identification the local component of
\(\chi=\chi_1\circ N_{L/L_1}\) is $\chi_{u_p}=\chi_{1,w_p}.$ As by construction, $\chi_{1,w_p}^2\neq 1$, 
$$
        {}^\theta\chi_{u_p}=\chi_{u_p}^{-1}\neq \chi_{u_p}.
$$
In particular, globally, $\chi^{-1}\neq \chi$. In order to prove the last claim of the above box, ${}^{\widetilde c}\chi=\chi$, just observe that \(\widetilde c\) acts trivially on \(L_1\), and that the norm \(N_{L/L_1}\) is invariant under automorphisms over \(L_1\).\\\\ With this Hecke character $\chi$ at hand, let
$$
        \Pi:=\operatorname{AI}_{L/E}(\chi)
$$
be the {\it automorphic induction} of $\chi$ to \(\GL_2(\mathbb A_E)\) as it has been established in the form needed here in Chp.\ 3, Thm.\ 6.2 of \cite{arthur-clozel} (see also \cite{henniart12}, Thm.\ 2). Since ${}^\theta\chi\neq\chi$, as we have just observed, \(\Pi\) is a cuspidal automorphic representation of \(\GL_2(\mathbb A_E)\) by  \cite{arthur-clozel}, Lem.\ 6.4 in Chp.\ 3. We claim that $\Pi$ is conjugate self-dual (with respect to $c$): Indeed, on the one hand, as ${}^\theta\chi=\chi^{-1}$,
$$ \Pi^\vee
        \cong
        \operatorname{AI}_{L/E}(\chi^{-1})
        =
        \operatorname{AI}_{L/E}({}^\theta\chi)
        \cong
        \Pi$$
but also, since  \( {}^{\widetilde c}\chi=\chi\), 
$$
        {}^c\Pi
        \cong
        \operatorname{AI}_{L/E}({}^{\widetilde c}\chi)
        =
        \operatorname{AI}_{L/E}(\chi)
      =
        \Pi
$$
whence in summary
$$
        {}^c\Pi\cong\Pi^\vee
$$
as claimed. At the archimedean places, the local extension \(L/E\) is split, and
automorphic induction gives the sum of the two characters obtained from
\(\chi_1\) and \({}^c\chi_1\). Therefore,
$$
        \Pi_\infty\cong \bigotimes_{\ww\in\Sigma_\infty}
        \left(z^{3/2}\bar z^{-3/2}\right)
        \times
        \left(z^{-3/2}\bar z^{3/2}\right),
$$
whence by \eqref{BCok} $\Pi_\infty\cong BC(\sigma_\infty)$ and we set as the first of our two desired globalizations
$$
\tau:=\Pi.       
$$
In order for this choice of a cuspidal automorphic representation $\tau$ of \(\GL_2(\mathbb A_E)\) to be suited to be the \(\GL_2(\mathbb A_E)\) -factor of a unitary cuspidal automorphic rerpesentation $\pi=\tau\hat\otimes\sigma$ of $L_2(\A_F)$, which matches all the conditions of Thm.\ \ref{thm:main} in the case of $\lambda_0=\gamma_2$, it is therefore now necessary and sufficient to check that $\tau$ descends to a cuspidal automorphic representation of $U(V_2,h)(\A_F)$ (and then to let $\sigma$ be this descent): In fact, if $\tau=BC(\sigma')$ with $\sigma'$ a cuspidal automorphic representation of $U(V_2,h)(\A_F)$, then the archimedean component $\sigma'_\infty$ is a tempered cohomological representation of $U(V_2,h)_\infty\cong U(1,1)\times \GL_2(\C)$ of the right highest weight determined by the fixed Kostant representative $w=(w_0,(w_0,w_0))$, as so was its local base change $\tau_\infty$. At the cost of changing the cohomological discrete series representation at the $U(1,1)$-factor of our previously fixed representation $\sigma_\infty$ from the beginning of this proof, we may assume that $\sigma_\infty=\sigma'_\infty$ and hence set $\sigma:=\sigma'$ globally. \\\\ 
We are hence left to show that such a cuspidal automorphic representation $\sigma'$ exists. For that, we first observe that by \cite{mok}, Cor.\ 2.5.9, $L^T(s,\Pi,\operatorname{As}^{-})$ has a pole at $s=1$, for a sufficiently big finite set of places $T$ of $F$. Therefore, $\Pi$ descends to $U(V_2,h)(\A_F)$ by \cite{mok}, Thm.\ 2.5.4, (see also his Rem.\ 2.5.5. Alternatively, the reader may have a look at \cite{zou}, Thm.\ 2.1 and Thm.\ 2.6 {\it ibidem}) i.e., there is an irreducible square-integrable automorphic representation $\sigma'$ of $U(V_2,h)(\A_F)$, whose stable base change is $\Pi=BC(\sigma')$. Such a representation $\sigma'$ is cuspidal as desired: Indeed, as the local Weil representation $\operatorname{Ind}_{W_{L_{u_p}}}^{W_{E_{\ww_p}}}\chi_{u_p}$ is irreducible, the local component $\Pi_{\ww_p}$, being the local automorphic induction, is supercuspidal. Consequently, since at a split non-archimedean place \(\vv_p\) of \(F\) below \(\ww_p\), we have $U(V_2,h)(F_{\vv_p})\cong \GL_2(F_{\vv_p})$, and hence local base change is the identity, the local component \(\sigma'_{\vv_p}\) is supercuspidal. However, a non-cuspidal square-integrable automorphic representation cannot have a tempered, and hence {\it a fortiori} supercuspidal, local component, cf.\ \cite{clozel2}, Prop.\ 4.10. Therefore \(\sigma'\) is cuspidal.  
\end{proof}


\begin{thebibliography}{99}

\bibitem[AGIKMS26]{AGIKMS} H. Atobe, W. T. Gan, A. Ichino, T. Kaletha, A. Mínguez, S. W. Shin, {\it Local Intertwining Relations and Co-tempered A-packets of Classical Groups}, arXiv:2410.13504v3 preprint (2026)


\bibitem[Art13]{arthur} J. Arthur, {\it The Endoscopic Classification of Representations - Orthogonal and Symplectic Groups}, Colloquium Publications {\bf 61} (AMS, 2013)

\bibitem[Art-Clo89]{arthur-clozel} J. Arthur, L. Clozel, {\it Simple algebras, base change, and the advanced theory of the trace formula}, Ann. Math. Studies, {\bf 120} (Princeton University Press, 1989)

\bibitem[AT09]{ArtinTate} E. Artin, J. Tate, \emph{Class Field Theory}, AMS Chelsea Publishing, (2009)

\bibitem[Baj-Cav24]{bajcav}, J. Bajpai, M. Cavicchi, {\it Relative Lie algebra cohomology of $SU(2,1)$ and Eisenstein classes on Picard surfaces}, arXiv:2402.00757v1 preprint (2024)

\bibitem[Bel-Che09]{bel-chen} J. Bella\"iche, G. Chenevier, Families of Galois representations and Selmer groups, {\it Ast\'erisque} {\bf 324} (2009)

\bibitem[Ber-VdG22]{Ber-vdG22} J. Bergström, G. van der Geer, Picard modular forms and the cohomology of local systems on a Picard modular surface, {\it Comment. Math. Helv.} {\bf 97} (2022) 305--381

\bibitem[Ber-Lap24]{BL} J. Bernstein, E. Lapid, On the meromorphic continuation of Eisenstein series, {\it J. Amer. Math. Soc.} {\bf 37} (2024) 187-234

\bibitem[Bor69]{borelbook} A. Borel, {\it Introduction aux groupes arithm\'etiques}, (Hermann, 1969)

\bibitem[Bor63]{bor1} A. Borel, Some finiteness properties of adele groups over number fields, {\it Publ. Math. IHES} {\bf 16} (1963) pp. 5--30

\bibitem[Bor-Cas83]{bocas} A. Borel, W. Casselman, $L^2$-Cohomology of locally symmetric manifolds of finite volume, {\it Duke Math. J.} \textbf{50} (1983) pp. 625-647

\bibitem[Bor-HCh62]{borelhch} A. Borel, Harish-Chandra, Arithmetic subgroups of algebraic groups, {\it Ann. Math.} {\bf 75} (1962) 485--535

\bibitem[Bor-Jac79]{bojac} A. Borel, H. Jacquet, {\it Automorphic forms and automorphic representations}, in: Proc. Sympos. Pure Math., Vol. XXXIII, part I, AMS, Providence, R.I., (1979), pp. 189-202

\bibitem[Bor-Ser73]{boserre} A. Borel, J. P. Serre, Corners and arithmetic groups, {\it Comm. Math. Helvet.} \textbf{48} (1973) 436--491

\bibitem[Bor-Wal00]{bowa} A. Borel, N. Wallach, {\it Continuous cohomology, discrete subgroups and representations of reductive groups}, Ann. of Math. Studies {\bf 94}, (Princeton Univ. Press, New Jersey, 2000)

\bibitem[BCGNT25]{BCGNT} G. Boxer, F. Calegari, T. Gee, J. Newton, J. Thorne, The Ramanujan and Sato--Tate conjectures for Bianchi modular forms, {\it Forum Math. Pi} \textbf{13} (2025)

\bibitem[Car12]{car} A. Caraiani, Local-global compatibility and the action of monodromy on nearby cycles, {\it Duke Math. J.} {\bf 161} (2012) 2311--2413

\bibitem[Che-Zou24]{zou} R. Chen, J. Zou, Arthur's multiplicity formula for even orthogonal and unitary groups, {\it JEMS} {\bf 27} (2024) 4769--4843

\bibitem[Clo93]{clozel2} L. Clozel, On the cohomology of Kottwitz's arithmetic varieties, {\it Duke Math. J.} {\bf 72} (1993) 757--795

\bibitem[Clo87]{clozelSLN} L. Clozel, On the cuspidal cohomology of arithmetic subgroups of \(\text{SL}(2n)\) and the first Betti number of arithmetic $3$-manifolds, {\it Duke Math. J.} {\bf 55} (1987) 475--486 

\bibitem[Clo91]{clozelihes} L. Clozel, Repr\'esentations galoisiennes associ\'ees aux repr\'esentations automorphes autoduales de $GL(n)$, {\it Publ. IHES} {\bf 73} (1991) 97--145

\bibitem[CKT26]{CKT} L. Clozel, A. Kret, O. Ta\"ibi, {\it Invariance Galoisienne des z\'eros centraux de fonctions L} with an appedix by O. Ta\"ibi and J.-L. Waldspurger,  arXiv:2602.09511v1 preprint (2026)

\bibitem[Cog07]{cogdel} J. W. Cogdell, {\it $L$-functions and Converse Theorems for ${\rm GL_n}$}, in: {\it Automorphic Forms and Applications}, IAS/Park City mathematics series, vol. 12, eds. P. Sarnak and F. Shahidi, AMS, Providence, R.I., (2007), pp. 95--177

\bibitem[Cog-PS-Sha11]{ckpssh} J. W. Cogdell, I. I. Piatetski-Shapiro, F. Shahidi, {\it Functoriality for the Quasisplit Classical Groups}, in: {\it On Certain L-functions}, Clay Mathematics Proceedings, vol. 13, eds. J. Arthur, J. W. Cogdell, S. Gelbart, D. Goldberg, D. Ramakrishnan, J.-K. Yu, (West Lafayette, IN, 2007) AMS, Providence, RI, 2011, pp. 117--140


\bibitem[Enr79]{enright} T. J. Enright, Relative Lie algebra cohomology and unitary representations of complex Lie groups, {\it Duke Math. J.} {\bf 46} (1979) 513--525

\bibitem[Fli-Zin95]{flickerzinoviev} Y. Z. Flicker, D. Zinoviev, On Poles of Twisted Tensor L-functions, {\it Proc. Japan Acad.} {\bf 71} (1995) 114--116

\bibitem[Fra08]{franke2} J. Franke, {\it A topological model for some summand of the Eisenstein cohomology of congruence subgroups}, in: {\it Eisenstein Series and Applications}, W.T. Gan, S.S. Kudla, Y. Tschinkel eds., {\it Progr. Math.} {\bf 258}, Birkh\"{a}user Boston, Boston, 2008, pp. 27--85


\bibitem[Fra-Sch98]{schwfr} J. Franke, J. Schwermer, A decomposition of spaces of automorphic forms, and the Eisenstein cohomology of arithmetic groups, {\it Math. Ann.}, \textbf{311} (1998), pp. 765-790


\bibitem[GGPW12]{ggp} W. T. Gan, B. H. Gross, D. Prasad, J.-L. Waldspurger, Sur les conjectures de Gross et Prasad, {\it Ast\'erisque} {\bf 346} (2012)


\bibitem[Got-Gro13]{gotsgro} G. Gotsbacher, H. Grobner, On the Eisenstein cohomology of odd orthogonal groups, {\it Forum Math.} {\bf 25} (2013) 283 --311

\bibitem[Grb-Sch21]{grabschwun} N. Grbac, J. Schwermer, Eisenstein series for rank one unitary groups and some cohomological applications, {\it Advances Math.} {\bf 376} (2021) Art. No. 107438

\bibitem[Grb-Sha15]{grbacshahidi} N. Grbac, F. Shahidi, Endoscopic transfer for unitary groups and holomorphy of Asai $L$-functions, {\it Pacific J. Math.} {\bf 276} (2015) 185--212 


\bibitem[Gro23]{grob_book} H. Grobner, {\it Smooth-automorphic forms and smooth-automorphic representations}, Series on Number Theory and Its Applications, {\bf 17} (World Scientific Pbl., 2023)


\bibitem[Gro10]{grob22} H. Grobner, Regular and Residual Eisenstein Series and the Automorphic Cohomology of $Sp(2,2)$, {\it Compos. Math.} {\bf 146} (2010) 21--57

\bibitem[Gro13]{grobner-EisRes} H. Grobner, Residues of Eisenstein series and the automorphic cohomology of reductive groups, {\it Compos. Math.} {\bf 149} (2013) 1061--1090


\bibitem[Gro-Har-Lin25]{GHL} H. Grobner, M. Harris, J. Lin, {\it Factorization of periods, construction of automorphic motives and Deligne's conjecture over CM-fields},  	arXiv:2509.02303 preprint (2025)


\bibitem[Gro-Rag14]{grob-ragh} H. Grobner, A. Raghuram, On some arithmetic properties of automorphic forms of $GL_m$ over a division algebra, {\it Int. J. Number Theory} {\bf 10} (2014) pp. 963--1013


\bibitem[Gro-\v Zun24]{grob_zun} H. Grobner, S. \v Zunar, On the notion of the parabolic and the cuspidal support of smooth-automorphic forms and smooth-automorphic representations, {\it Monatshefte Math.} {\bf 204} (2024) 455 -- 500 


\bibitem[Hrd87]{harderGU} G. Harder, Eisensteinkohomologie für Gruppen vom Typ $GU(2,1)$, {\it Math. Ann.} {\bf 278} (1987) 563--592


\bibitem[Har-Lab04]{harris-labesse} M. Harris, J.-P. Labesse, Conditional Base Change for Unitary Groups, {\it Asian J. Math.} {\bf 8} (2004) 653--684


\bibitem[Har-Tay01]{HT} M. Harris, R. Taylor, {\it The geometry and cohomology of some simple Shimura varieties}, Ann. of Math. Studies {\bf 151}, (Princeton Univ. Press, New Jersey, 2001)


\bibitem[Hay-Sch05]{hayataschwermer} T. Hayata, J. Schwermer, On arithmetic subgroups of a $Q$-rank 2 form of $SU(2,2)$ and their automorphic cohomology, {\it J. Math. Soc. Japan} {\bf 57} (2005) 357--385 


\bibitem[Hen10]{henniartAsai} G. Henniart, Correspondance de Langlands et fonctions L des carr\'es ext\'erieur et sym\'etrique, {\it Int. Math. Res. Not.}  {\bf 4} (2010) 633--673

\bibitem[Hen12]{henniart12} G. Henniart, Induction automorphe globale pour les corps des nombres, {\it Bull. Soc. math. France} {\bf 140} (2012) 1--17

\bibitem[Hen00]{henniart} G. Henniart, Une preuve simple des conjectures de Langlands pour $GL(n)$ sur un corps $p$-adique, {\it Invent. Math.} \textbf{139} (2000) 439--455

\bibitem[Joh90]{johnson} J. F. Johnson, Stable base change $\C/\R$ of certain derived functor modules, {\it Math. Ann.} {\bf 287} (1990) 467--493

\bibitem[KMSW14]{KMSW} T. Kaletha, A. Minguez, S. W. Shin, P.-J. White, {\it Endoscopic classification of representations: inner forms of unitary groups}, preprint (2014).

\bibitem[Kna86]{knappbased} A. W. Knapp, \emph{Representation Theory of Semisimple Groups - An Overview Based On Examples}, Princeton Univ. Press, (1986)


\bibitem[Kna-Vog95]{knappvogan} A. W. Knapp, D. A. Vogan Jr. \emph{Cohomological induction and unitary representations}, Princeton Univ. Press, (1995)

\bibitem[Kim-Kri04]{kim-krish04} H. Kim, M. Krishnamurthy, Base change lift for odd unitary groups. Functional analysis VIII, {\it Various Publ. Ser. (Aarhus)}, {\bf 47} (Aarhus Univ., Aarhus, 2004), 116--125

\bibitem[Kim-Kri05]{kim-krish05} H. Kim, M. Krishnamurthy, Stable base change lift from unitary groups to $\GL_n$, {\it Int. Math. Res. Pap.} {\bf 1} (2005) 1--52

\bibitem[Kos61]{kostant} B. Kostant, Lie algebra cohomology and the generalized Borel-Weil theorem, {\it Ann. Math.} \textbf{74} (1961) 329--387


\bibitem[Kud94]{kudla} S. S. Kudla, {\it Local Langlands correspondence: the Non-Archimedean Case}, in: {\it Motives}, Proc. Sympos. Pure Math., Vol. LV, Part II, Amer. Math. Soc. (1994) 365--392

\bibitem[Lab11]{lab}   J.-P. Labesse, {\it Changement de base CM et s\'eries discr\`etes}, in: {\it On the Stabilization of the Trace Formula}, Vol.\ I, eds.\ L.\ Clozel, M.\ Harris, J.-P.\ Labesse, B.-C.\ Ng\^ o, International Press, Boston, MA, 2011, pp. 429--470


\bibitem[Lan76]{langlandsLNM} R. P. Langlands, {\it On the Functional Equations Satisified by Eisenstein Series},  LNM {\bf 544} (1976)


\bibitem[Man-Ome-Yan26]{MOY} N. Matringe, O. Offen, C. Yang, {\it Intertwining periods, L-functions and local-global principles for distinction of automorphic representations}, arXiv:2509.00441v2 preprint (2026)


\bibitem[M\oe-Wal89]{MW}  C. M\oe glin, J.-L. Waldspurger, Le spectre r\'esiduel de GL(n),  {\it Ann. Sci. \'Ecole Norm. Sup.} {\bf 22} (1989) 605--674.

\bibitem[M\oe-Wal95]{moewal} C. M\oe glin, J.-L. Waldspurger, {\it Spectral decomposition and Eisenstein series}, Cambridge Univ. Press (1995)

\bibitem[Mok15]{mok}  C. P. Mok, {\it Endoscopic classification of representations of quasi-split unitary groups}, Memoirs of the AMS {\bf 235} (2015).

\bibitem[Mor10]{morel} S. Morel, {\it On the cohomology of certain non-compact Shimura varieties}, Ann. Math. Studies {\bf 173} (Princeton Univ. Press, New Jersey, 2010)

\bibitem[Mun26]{mundy} S. Mundy, {\it A vanishing theorem for residual Eisenstein cohomology}, arXiv:2603.12472v1 preprint (2026)

\bibitem[Nai-Pra21]{nairprasad} A. N. Nair, D. Prasad, Cohomological representations for real reductive groups, {\it J. London Math. Soc.} {\bf 104} (2021) 1515--1571

\bibitem[Nai-Rai23]{nairrai} A. N. Nair, A. Rai, Automorphic Lefschetz properties for noncompact arithmetic manifolds, {\it J. Inst. Math. Jussieu} {\bf 22} (2023) 1655--1702


\bibitem[Roh96]{rohlfs} J. Rohlfs, Projective limits of locally symmetric spaces and cohomology, {\it J. Reine Angew. Math.} \textbf{479} (1996) pp. 149--182

\bibitem[Roh-Spe11]{rosp} J. Rohlfs, B. Speh, {\it Pseudo Eisenstein Forms and the Cohomology of Arithmetic Groups III: Residual Cohomology Classes}, in: {\it On certain $L$-functions: Conference in Honor of Freydoon Shahidi}, eds.\ J.\ Arthur, J.\ W.\ Codgell, S.\ Gelbart, D.\ Goldberg, D.\ Ramakrishnan, J.-K.\ Yu, AMS, Providence, R.I., (2011), pp. 501--524


\bibitem[Sal98]{salamribsu} S. Salamanca-Riba, On the unitary dual of real reductive Lie groups and the $A_\q(\lambda)$ modules: The strongly regular case, {\it Duke Math. J.} {\bf 96} (1998) 521--546


\bibitem[Ser79]{serrearith} J.-P. Serre, {\it Arithmetic Groups}, in: {\it Homological Group Theory} London Math. Soc. Lecture Note Series {\bf 36} (Cambridge University Press, 1979)


\bibitem[Shi14]{shin} S. W. Shin, On the cohomological base change for unitary similitude groups, appendix to: W. Goldring, {\it Galois representations associated to holomorphic limits of discrete series I: Unitary Groups}, {\it Compos. Math.} {\bf 150} (2014) 191--228

\bibitem[Sha88]{shahidi2} F. Shahidi, On the Ramanujan conjecture and finiteness of poles for certain $L$-functions, {\it Ann. Math.} \textbf{127} (1988) 547--584

\bibitem[Shl74]{shal} J. A. Shalika, The multiplicity one theorem for $GL_n$, {\it Ann. Math.} \textbf{100} (1974) 171--193

\bibitem[Sil94]{Sil94} J. H. Silverman, \emph{Advanced Topics in the Arithmetic of Elliptic Curves}, Graduate Texts in Mathematics, {\bf 151} (1994)

\bibitem[Spe83]{speh} B. Speh, Unitary representations of $GL(n,\R)$ with nontrivial $(\g,K)$-cohomology, {\it Invent. Math.} {\bf 71} (1983) 443--465

\bibitem[Spe81]{speh81} B. Speh, {\it Unitary representations of SL$(\R)$ and the cohomology of congruence subgroups}, in: {\it Non Commutative Harmonic Analysis and Lie Groups}, Lecture Notes in Mathematics {\bf 880} (1981) 483--505


\bibitem[Vog97]{voganupq}, D. A. Vogan Jr., {\it Cohomology and group representations}, in: Proc. Sympos. Pure Math., Vol. LXI, part I, AMS, Providence, R.I., (1997), pp. 219--243

\bibitem[Vog84]{voganunit} D. A. Vogan Jr., Unitarizability of certain series of representations, {\it Ann. of Math.} {\bf 120} (1984) 51-90


\bibitem[Vog-Zuc84]{vozu} D. A. Vogan Jr., G. J. Zuckerman, Unitary representations with nonzero cohomology, {\it Comp. Math.} \textbf{53} (1984) 51--90


\bibitem[Wal84]{wallach} N. Wallach, {\it On the constant term of a square-integrable automorphic form}, in: {\it Operator algebras and group representations}, Vol. II (Neptun, 1980), {\it Monographs Stud. Math.} \textbf{18} (1984), pp. 227-237


\end{thebibliography}
\end{document}